\documentclass[11pt,letterpaper]{amsart}

\numberwithin{equation}{section}
\theoremstyle{plain}
\newtheorem{theorem}{Theorem}[section]
\newtheorem{prop}[theorem]{Proposition}

\newtheorem{lemma}[theorem]{Lemma}

\newtheorem{corollary}[theorem]{Corollary}

\theoremstyle{definition}

\usepackage{hyperref}

\newtheorem{remark}[theorem]{Remark}

\newcommand{\Rmnum}[1]{\expandafter\@slowromancap\romannumeral #1@}

\newcommand{\mr}{\mathbb{R}}
\newcommand{\ud}{\mathrm{d}}
\newcommand{\ms}{\mathbb{S}}
\newcommand{\fint}{-\mkern -19mu\int}
\allowdisplaybreaks

\address{Mingxiang Li, Department  of Mathematics \& Institue of Mathematical Sciences, Chinese University of Hong Kong, Shatin, NT, Hong Kong}
\email{mingxiangli@cuhk.edu.hk}

\begin{document}
	\title{Conformal metrics  with finite total Q-curvature revisited}
	\author{Mingxiang Li}
\date{}
\maketitle

	\begin{abstract}
		Given  a conformal metric with finite total Q-curvature, we show  that the  assumptions on scalar curvature sensitively govern the Q-curvature integral. Additionally, we  introduce a conformal mass for such manifolds.  Using such mass, we provides a necessary and sufficient condition for the metric to be normal without assuming metric completeness. As applications, we derive volume comparison theorems and prove a positive mass type theorem related to Q-curvature.
	\end{abstract}

	\section{Introduction}

	For compact surfaces, the celebrated Chern-Gauss-Bonnet formula \cite{CHern} intricately links the surface's geometry to its topology. Analogously, for complete open surfaces, this connection is elucidated by the works of Cohn-Vossen \cite{CV},  Huber \cite{Hu}, and Finn \cite{Fin}. For any complete open surface $(M,g)$ with integrable negative part of Gaussian curvature $K_g$, the following inequality holds
	\begin{equation}\label{CV inequ}
		\int_MK_g\ud v_g \leq 2\pi\chi(M)
	\end{equation}
	where  $\chi(M)$ is the  Euler characteristic of $M$. We also suggest  interested readers referring to \cite{Li-Tam} and \cite{Muller-Sverak} for additional insights. 
	
	When it comes to higher-dimensional manifolds, is there a similar connection to explore?   This is the  Problem 11 in chapter 7 of \cite{Schoen-Yau} established by Schoen and Yau.  
	Some progress has been made by  Walter \cite{Walter}, Poor \cite{Poor},  and Greene-Wu \cite{GW} under the assumption of non-negative sectional curvature.  For complete Riemannian manifolds $(M,g)$  with non-negative Ricci curvature, Yau proposed a significant conjecture in \cite{Schoen-Yau}, stating that
	\begin{equation}\label{Yau's conjecture}
		\limsup_{r\to\infty} r^{2-n} \int_{B^g_r(p)} R_g dv_g < +\infty
	\end{equation}
	where $R_g$ denotes  the scalar curvature and $B^g_r(p)$ represents the geodesic ball centered at point $p$  with radius  $r$.  This conjecture can be interpreted as a higher-dimensional analogue of the classical Cohn-Vossen inequality \eqref{CV inequ}. Partial progress toward resolving this conjecture has been achieved in the works of \cite{Xu gy} and \cite{Zhu}.

	In higher-dimensional manifolds, Q-curvature    introduced by Branson \cite{Bra} is  a natural generalization of Gaussian curvature. In particular,  for a four-dimensional manifold $(M,g)$, Q-curvature is 
	defined by
	$$Q_g:=-\frac{1}{6}(\Delta_gR_g-R_g^2+3|Ric_g|^2)$$
	where $\Delta_g$ is the Laplace-Beltrami operator,   $Ric_g$ and $R_g$ are the Ricci curvature and scalar curvature  of $(M,g)$ respectively. For closed four-dimensional manifold $(M,g)$, the Chern-Gauss-Bonnet formula  can be stated as follows
	$$\int_M\left(|W|^2+4Q_g\right)\ud\mu_g=32\pi^2 \chi(M)$$
	where $W$ is the Weyl tensor. Combining with the Paneitz operator \cite{Paneitz}, for a conformal metric $\tilde g=e^{2u}g$, there holds a conformally invariant equation
	$$
P_gu+Q_g= Q_{\tilde g}e^{4u}
$$
where the Paneitz operator $P_g$ is defined by $$P_g=\Delta_g^2-\mathrm{div}_g\left((\frac{2}{3}R_gg-2Ric_g)d\right).$$	
In particular, for a conformal metric $g=e^{2u}|dx|^2$ on $\mr^n$ where $n\geq 4$ is an even integer, the Q-curvature  satisfies 
the conformally invariant equation:
\begin{equation}\label{equ:conformal eqution}
	(-\Delta)^{\frac{n}{2}}u=Qe^{nu}.
\end{equation}
More information  about the equation \eqref{equ:conformal eqution} can be found in \cite{B-H-S},  \cite{CQY2},  \cite{CL 91 Duke}, \cite{Cheng-Lin}, \cite{Li23}, \cite{Li 23 Q-curvature}, \cite{Li 24 Obstruction}, \cite{LW}, \cite{Lin}, \cite{Mar MZ}, \cite{SW}, \cite{Wang IMRN},  \cite{Wang Yi},  \cite{WX}, \cite{Xu05},  and the references therein.

With the assistance of Q-curvature, Chang, Qing, and Yang \cite{CQY} established an analogous Cohn-Vossen inequality \eqref{CV inequ} for a complete and non-compact manifold $(\mathbb{R}^4, e^{2u}|dx|^2)$. Fang \cite{Fa} and   Ndiaye-Xiao  \cite{NX} further extended this to higher dimensional cases.  For the reader's convenience, we repeat their results as follows.
	\begin{theorem}\label{thm: CQY fang NX}(See \cite{CQY}, \cite{Fa}, \cite{NX})
		Given a complete and  conformal metric $g=e^{2u}|dx|^2$ on $\mr^n$ where $n\geq 4$ is an even integer with finite total Q-curvature. Suppose that  
		the  scalar curvature $R_g\geq 0$ near infinity. Then there holds
		\begin{equation}\label{equ:Qenu leq 1/2}
			\int_{\mr^n}Qe^{nu}\ud x\leq \frac{(n-1)!|\mathbb{S}^n|}{2}
		\end{equation}
		where $|\ms^n|$ denotes the volume of standard n-sphere.
	\end{theorem}
We say the metric $g$ has  finite total Q-curvature if $Qe^{nu}\in L^1(\mr^n)$ and  “near infinity" means outside a compact set. 
For brevity, we consider the smooth metrics and even dimensional cases throughout this paper. 

A crucial  technique used in \cite{CQY} is  “normal metric" which is introduced by Finn \cite{Fin} in two-dimensional case. Precisely,
we say the solution $u(x)$ to \eqref{equ:conformal eqution}  with $Qe^{nu}\in L^1(\mr^n)$ is  normal if $u(x)$ satisfies the integral equation
	\begin{equation}\label{integral equation}
		u(x)=\frac{2}{(n-1)!|\mathbb{S}^n|}\int_{\mr^n}\log\frac{|y|}{|x-y|}Q(y)e^{nu(y)}\ud y+C
	\end{equation}
	where $C$ is a constant which may be different from line to line throughout this paper.
	In particular, the metric $g=e^{2u}|dx|^2$ with finite total Q-curvature is said to be  normal if $u(x)$ is normal. 
		 Indeed, for any complete and  normal metric, the inequality \eqref{equ:Qenu leq 1/2} holds,  as demonstrated in Theorem 1.3 of \cite{CQY}. The geometric condition $R_g\geq 0$ near infinity plays an important role since it ensures that the metric is normal.   To ensure normal metrics, some other geometric conditions are explored in \cite{Li 23 Q-curvature}, \cite{WW}, \cite{Zhang}.

The scalar curvature $R_g$ is a very subtle geometric quantity. In Chang-Qing-Yang's work \cite{CQY}, the non-negative $R_g$ near infinity is solely used to ensure the normal metric. It is interesting to ask whether exists other geometric control under the same assumption on scalar curvature. Fortunately, we find that the scalar curvature $R_g\geq 0$ near infinity also  gives a sharp lower bound of Q-curvature integral which is very different from two-dimensional case. Furthermore, if the scalar curvature demonstrates controlled decay behavior near infinity, we can establish more restrictive  estimates for the Q-curvature integral.

\begin{theorem}\label{thm: complete R_g geq 0}
	Given a complete and   conformal metric $g=e^{2u}|dx|^2$ on $\mr^n$ where $n\geq 4$ is an even integer with finite total Q-curvature. 
	\begin{enumerate}
		\item Suppose that  
		the  scalar curvature $R_g\geq 0$ near infinity. Then there holds
		$$0\leq \int_{\mr^n}Qe^{nu}\ud x\leq \frac{(n-1)!|\mathbb{S}^n|}{2}.$$
		\item Suppose that  
		the  scalar curvature $R_g(x)\geq C|x|^{-p}$ near infinity for some constant $C>0$ and $p\geq 0$. Then there holds
		$$\frac{(n-1)!|\mathbb{S}^n|}{2}\max\{0, 1-\frac{p}{2}\}\leq \int_{\mr^n}Qe^{nu}\ud x\leq \frac{(n-1)!|\mathbb{S}^n|}{2}.$$
		In particular, if the scalar curvature $R_g\geq C>0$ near infinity, there holds
		$$ \int_{\mr^n}Qe^{nu}\ud x=\frac{(n-1)!|\mathbb{S}^n|}{2}.$$
	\end{enumerate}
\end{theorem}

Drawing on the aforementioned findings and the result in \cite{CQY2}, we can effectively control  the integral of the Q-curvature over four-manifolds with finitely  simple ends. A detailed definition of simple ends can be found in Section \ref{sec: simple ends}.
\begin{theorem}\label{thm: chim(M) R_g geq 0}
	Given  a complete four-dimensional manifold $(M, g)$ with finitely many  conformally flat simple ends and finite total Q-curvature. 
		\begin{enumerate}
		\item 	Suppose that the scalar curvature is non-negative at each end. Then there holds
		$$\chi(M)-\#\{ends\}\leq \frac{1}{32\pi^2}\int_M\left(|W|^2+4Q_g\right)\ud\mu_g\leq \chi(M).$$
		\item  Suppose that the scalar curvature $R_g\geq C$ for some constant $C>0$ at each end. Then there holds
		$$ \int_M\left(|W|^2+4Q_g\right)\ud\mu_g=32\pi^2 \chi(M).$$
	\end{enumerate}
	Here,   $\chi(M)$, $W$ and $\#\{ends\}$ denote the Euler number, Weyl  tensor and the number of  ends respectively.
\end{theorem}

Another crucial geometric assumption in Theorem \ref{thm: CQY fang NX} is the metric completeness condition, which is essential for establishing the upper bound of the Q-curvature integral in \eqref{equ:Qenu leq 1/2}. Without this completeness requirement, the inequality \eqref{equ:Qenu leq 1/2} fails to hold in general. For example, consider the conformal metric  $g_0=e^{2u_0}|dx|^2$ where $u_0(x)=\log\frac{2}{1+|x|^2}$.

 Interestingly, we find that the non-negative scalar curvature condition alone - even without metric completeness - still yields effective and sharp control of the Q-curvature integral. Furthermore, we obtain optimal decay estimates for the lower bound of the conformal metric. This is different from the existing results in \cite{Ma-Qing}, \cite{Hyder-Mancini-Mart}, \cite{Li 23 Q-curvature}, and \cite{Zhang}, which  require stronger  curvature assumptions (such as non-negative Ricci curvature or non-negative Q-curvature near infinity) to achieve comparable estimates. 

\begin{theorem}\label{thm: R_g geq 0}
		Given a conformal metric $g=e^{2u}|dx|^2$ on $\mr^n$ where $n\geq 4$ is an even integer with finite total Q-curvature. Suppose that the  scalar curvature $R_g\geq 0$ near infinity.  Then there holds
		\begin{enumerate}
			\item\label{thm:0 leq alpha_0 leq 2}
			 $$0\leq \int_{\mr^n}Qe^{nu}\ud x\leq (n-1)!|\mathbb{S}^n|,$$
			\item\label{thm:u/log x}
			$$\lim_{|x|\to\infty}\inf\frac{u(x)}{\log|x|}= -\frac{2}{(n-1)!|\mathbb{S}^n|}\int_{\mr^n}Qe^{nu}\ud x.$$
		\end{enumerate}
\end{theorem}

	For compact four-dimensional manifolds, Gursky (Theorem B in  \cite{Gursky}) established sharp estimates on the Q-curvature integral under the assumption of a non-negative Yamabe invariant.  Part  \eqref{thm:0 leq alpha_0 leq 2} of Theorem \ref{thm: R_g geq 0} can be viewed as a counterpart  in the context of non-compact manifolds.  As a direct consequence of Part \eqref{thm:u/log x}, we obtain the following corollary. In fact, it is also obtained in \cite{Li 23 Q-curvature} using the distance growth identity (see Theorem 1.4 in \cite{Li 23 Q-curvature}).
	\begin{corollary}
		Given a conformal metric $g=e^{2u}|dx|^2$ on $\mr^n$ where $n\geq 4$ is an even integer with finite total Q-curvature. Suppose that the  scalar curvature $R_g\geq 0$ near infinity. 
	If there holds $\int_{\mr^n}Qe^{nu}\ud x<\frac{(n-1)!|\mathbb{S}|}{2},$  then the  metric $g$ is complete. 
	\end{corollary}

 Meanwhile, with the  help of Theorem \ref{thm: R_g geq 0} and the classification theorem  for $Q\equiv(n-1)!$ in \eqref{equ:conformal eqution} established in \cite{CY MRL}, \cite{Lin}, \cite{Mar MZ}, \cite{WX}, and \cite{Xu05}, we are able to confirm  Conjecture 2 in \cite{LW}.
\begin{theorem}\label{thm: Volume comparison Q geq n-1!}
	Given a conformal metric $g=e^{2u}|dx|^2$ on $\mr^n$ where $n\geq 4$ is an even integer with finite total Q-curvature. Supposing that  the scalar curvature $R_g\geq 0$ near infinity and the Q-curvature satisfies 
	$$Q\geq (n-1)!,$$
	then there holds
	$$\int_{\mr^n}e^{nu}\ud x\leq |\ms^n|$$
	with equality holds if and only if $u(x)=\log \frac{2\lambda}{\lambda^2+|x-x_0|^2}$ for some constant $\lambda>0$ and $x_0\in \mr^n$.
\end{theorem}

The non-negativity condition of scalar curvature  $R_g$ near infinity plays an essential role in Theorem \ref{thm: R_g geq 0}. This requirement cannot be relaxed to $$R_g\geq -C$$ for some positive constant $C$, as such a weaker condition would generally fail to guarantee control over the Q-curvature integral. To illustrate this, let us consider the specific example where $u_1(x)=-\beta\log(|x|^2+1)+|x|^2$ where $\beta \in \mr$ and examine the metric $g_1=e^{2u_1}|dx|^2$. Through direct computation, we can verify that: $(-\Delta)^{\frac{n}{2}}u_1\in L^1(\mr^n)$, $R_{g_1}\geq -C_1$ for some constant $C_1>0$ and  $$\int_{\mr^n}(-\Delta)^{\frac{n}{2}}u_1\ud x=(n-1)!|\mathbb{S}^n|\beta. $$

If we additionally assume that the scalar curvature is bounded from above or the volume is finite, we will obtain more restrictive control of Q-curvature integral.
\begin{theorem}\label{thm:scalar curvature bouended}
		Given a conformal metric $g=e^{2u}|dx|^2$ on $\mr^n$ where $n\geq 4$ is an even integer with finite total Q-curvature. 
		\begin{enumerate}
			\item 
			Suppose that the  scalar curvature $0\leq R_g\leq C$ near infinity for some positive constant $C$. Then there holds
			$$0\leq \int_{\mr^n}Qe^{nu}\ud x\leq \frac{(n-1)!|\mathbb{S}^n|}{2}\; \mathrm{or}\;  \int_{\mr^n}Qe^{nu}\ud x=(n-1)!|\mathbb{S}^n|.$$
			\item\label{part 2 of R-g leq C} Suppose that the scalar curvature $0\leq R_g\leq C$ near infinity. If the volume is finite i.e. $e^{nu}\in L^1(\mr^n)$, we have 
			$$\int_{\mr^n}Qe^{nu}\ud x=(n-1)!|\mathbb{S}^n|.$$
		\end{enumerate}
\end{theorem}
As an application of Part \eqref{part 2 of R-g leq C} of Theorem \ref{thm:scalar curvature bouended}, we can  answer  Conjecture 1 in \cite{LW} under the additional assumption that the scalar curvature $R_g$ is bounded from above.

\begin{theorem}\label{thm: Volume comparison Q leq n-1!}
	Given a conformal metric $g=e^{2u}|dx|^2$ on $\mr^n$ where $n\geq 4$ is an even integer with finite total Q-curvature. Supposing that  the scalar curvature $0\leq R_g\leq C$ near infinity and the Q-curvature satisfies 
	$$Q\leq (n-1)!,$$
	then there holds
	\begin{equation}\label{volume lower bound}
		\int_{\mr^n}e^{nu}\ud x\geq |\ms^n|
	\end{equation}
	with equality holds if and only if $u(x)=\log \frac{2\lambda}{\lambda^2+|x-x_0|^2}$ for some constant $\lambda>0$ and $x_0\in \mr^n$.
\end{theorem}  

In the  work \cite{Li 23 Q-curvature}, a volume entropy for conformal metrics $g=e^{2u}|dx|^2$ on $\mr^n$ is introduced , defined as 
\begin{equation}\label{volume entropy}
	\tau(g):=\lim_{r\to\infty}\sup\frac{\log V_g(B_r(0))}{\log|B_r(0)|}
\end{equation}
where $B_r(0)$ denotes the Euclidean ball centered at origin with radius $r$, while   $V_g(B_r(0))$ and $|B_r(0)|$ represents the volume of $B_r(0)$ with respect to  the conformal metric $g$ and Euclidean metric,  respectively.
The following theorem is established in \cite{Li 23 Q-curvature}.
\begin{theorem}\label{thm: volume entropy}(Theorem 1.1 in \cite{Li 23 Q-curvature})
	Given a complete conformal metric $g=e^{2u}|dx|^2$ on $\mr^n$ where $n\geq 4$ is an even integer with finite total Q-curvature.  Then the conformal  metric  $g$   is normal if and only if the volume entropy $\tau(g)$ is finite.  Moreover, if $g$ is normal, there holds
	\begin{equation}\label{volume entropy identity}
		\tau(g)=1-\frac{2}{(n-1)!|\mathbb{S}^n|}\int_{\mr^n}Qe^{nu}\ud x.
	\end{equation}
\end{theorem}

With the aid of Theorem \ref{thm: complete R_g geq 0} and Theorem \ref{thm: volume entropy}, we can attain volume growth control under the condition of scalar curvature decay control. For general Riemannian manifolds with non-negative Ricci curvature, certain volume growth controls under scalar curvature constraints have been obtained in \cite{CLS} and \cite{MW}.
\begin{theorem}\label{thm: volume growth under scalar curvature decay}
		Given a complete  and conformal metric $g=e^{2u}|dx|^2$ on $\mr^n$ where $n\geq 4$ is an even integer with finite total Q-curvature.  If the  scalar curvature $R_g(x)\geq C|x|^{-p}$ near infinity for some constant $C>0$ and $ p\geq 0$. Then the volume entropy $\tau(g)$ satisfies 
		$$\tau(g)\leq \min\{ 1, \frac{p}{2}\}.$$
\end{theorem}

  For dimensions $n\geq 4$, the finiteness  of $\tau(g)$ alone is insufficient to guarantee the normal metric when metric completeness is not assumed. For example, consider $u(x)=-|x|^2$.    Interestingly, as shown in \cite{CQY}, metric completeness is not used in establishing that the scalar curvature condition $R_g\geq 0$ near infinity implies  the normal metric.  This observation naturally leads to the following question: Can we identify a geometric quantity, analogous to volume entropy, that characterizes normal metrics without assuming metric completeness?
  
To address this question, we introduce a mass-type quantity for the conformal metric $g=e^{2u}|dx|^2$, defined as follows 
  \begin{equation}\label{conformal mass}
  	m_{c}(g):=\frac{1}{(n-1)(n-2)|\mathbb{S}^{n-1}|}\inf_{x\in \mr^n}\lim_{r\to\infty}\inf r^{3-n}\int_{\partial B_r(x)}R_ge^{2u}\ud \sigma
  \end{equation}
where $\ud \sigma$ is the area element on $\partial B_r(x)$ respect to Euclidean metric.
When $m_c(g)$ is well-defined, for any positive constant $C>0$, it exhibits the scaling invariant property:  $m_c(Cg)=m(g)$. Thus, we shall refer to this quantity as the “conformal mass". 

In fact, such a mass quantity is primarily motivated by the foundational work of the well-known positive mass theorem in mathematical physics. This area has garnered significant attention through numerous seminal contributions, including but not limited to \cite{SY CMP}, \cite{Witten}, \cite{Bart},  and many other pivotal studies.  For  an asymptotically flat manifold $(M,g)$, the classical definition of ADM mass \cite{ADM} is defined  as 
\begin{equation}\label{ADM def1}
	m_{ADM}(g):=\frac{1}{2(n-1)|\mathbb{S}^{n-1}|}\lim_{r\to\infty}\int_{\partial B_r}(g_{ij,j}-g_{jj,i})\nu^i\ud \sigma
\end{equation}
where $\nu$ denotes the outward unit normal vector. An alternative intrinsic formulation utilizing the Ricci tensor is presented in \cite{AH}, \cite{Chr}:
\begin{equation}\label{ADM def2}
	m_{ADM}(g)=\frac{1}{(n-1)(n-2)|\mathbb{S}^{n-1}|}\lim_{r\to\infty}\int_{\partial B_r}\left(Ric_g-\frac{1}{2}R_gg\right)(X,\nu_g)\ud \sigma_g
\end{equation}
where $X$ is the Euclidean conformal Killing vector field $x_i\frac{\partial}{\partial x_i}$. Readers interested in further details are referred to \cite{Her}, \cite{Miao-Tam}, and the comprehensive literature therein. Recent developments in the study of positive energy theorems have extended to the context of Q-curvature on asymptotically flat manifolds, as demonstrated in the works of \cite{ALL} and \cite{ALM}.

Using the conformal mass $m_c(g)$,  we can answer the previously raised question.
\begin{theorem}\label{thm: conformal mass}
	Given a conformal metric $g=e^{2u}|dx|^2$ on $\mr^n$ where $n\geq 4$ is an even integer with finite total Q-curvature. 
	\begin{enumerate}
		\item The metric $g$ is normal if and only if $m_{c}(g)>-\infty$. 
		\item  If the metric $g$ is normal, there holds
		$$m_{c}(g)=\frac{2}{(n-1)!|\mathbb{S}^n|}\int_{\mr^n}Qe^{nu}\ud x\cdot \left(2-\frac{2}{(n-1)!|\mathbb{S}^n|}\int_{\mr^n}Qe^{nu}\ud x\right).$$ 
	\end{enumerate}
\end{theorem}
\begin{remark}
	In the above theorem, we do not need to give any restriction on the upper bound for $m_c(g)$. In fact, using the precise representation  of $m_c(g)$, it is not hard to see that $m_c(g)\leq 1$ holds automatically.
\end{remark}
As an application,  we have the following rigidity theorem for conformal mass.
\begin{theorem}\label{thm:positive mass}
	Given a complete normal metric $g=e^{2u}|dx|^2$ on $\mr^n$ where $n\geq 4$ is an even integer with finite total Q-curvature. If the  Q-curvature satisfies  $Q\geq 0$,  then the conformal mass satisfies
	$$m_c(g)\geq 0$$
	with the  equality holds if and only if $u$ is a constant.
\end{theorem}

When imposing a constraint on the integral of scalar curvature, we also obtain similar control of the Q-curvature integral. We set the notations $\varphi^+=\max\{\varphi, 0\}$ and $\varphi^-=\max\{-\varphi,0\}$.

\begin{theorem}\label{thm: integral R_g}
	Given a  conformal metric $g=e^{2u}|dx|^2$ on $\mr^n$ where $n\geq 4$ is an even integer with finite total Q-curvature. \begin{enumerate}
		\item\label{thm:int R^-n/2} Suppose that  
		the  negative part of scalar curvature $R_g^-$  satisfies 
		\begin{equation}\label{equ: integral R_g^-}
			\int_{\mr^n}(R^-_g)^{\frac{n}{2}}e^{nu}\ud x<+\infty.
		\end{equation}
		Then there holds
		$$0\leq \int_{\mr^n}Qe^{nu}\ud x\leq (n-1)!|\mathbb{S}^n|.$$
		\item \label{thm:int R^n/2}Suppose that   scalar curvature $R_g$  satisfies 
		\begin{equation}\label{equ: integral R_g}
			\int_{\mr^n}|R_g|^{\frac{n}{2}}e^{nu}\ud x<+\infty.
		\end{equation}
		Then there holds
		$$ \int_{\mr^n}Qe^{nu}\ud x=0 \quad \mathrm{or}\quad \int_{\mr^n}Qe^{nu}\ud x= (n-1)!|\mathbb{S}^n|.$$
		Moreover, if the volume is finite  i.e. $e^{nu}\in L^1(\mr^n)$  in addition, then we have 
		 $$\int_{\mr^n}Qe^{nu}\ud x= (n-1)!|\mathbb{S}^n|.$$
	\end{enumerate}
\end{theorem}  

\begin{remark}
		Part \eqref{thm:int R^-n/2}  in Theorem \ref{thm: integral R_g} is a generalization  of  Theorem \ref{thm: R_g geq 0}.    Part \eqref{thm:int R^n/2} is also obtained  by  Lu-Wang \cite{Lu-Wang}, Wang-Zhou \cite{Wang-Zhou},  and Li-Wang \cite{Li-Wang} with the help of very different methods.
\end{remark}

Now, let us briefly outline the structure of this paper. In Section \ref{section:  estimate for normal }, we establish some crucial  estimates for normal solutions. Leveraging these estimates, we proceed to prove Theorem \ref{thm: complete R_g geq 0},  Theorem \ref{thm: R_g geq 0}, Theorem \ref{thm: Volume comparison Q geq n-1!}, Theorem \ref{thm:scalar curvature bouended}, Theorem \ref{thm: Volume comparison Q leq n-1!},  Theorem \ref{thm: volume growth under scalar curvature decay}, and  Theorem \ref{thm: integral R_g} in Section \ref{section: scalar curvature and Q-curvature integral}.
Later, in Section \ref{sec: simple ends}, we study the complete four-dimensional manifolds with simple ends and provide the proof of Theorem \ref{thm: chim(M) R_g geq 0}.  In Section \ref{sec: conformal mass}, making use of conformal mass, we  establish   the sufficient and necessary condition for normal metric  and prove Theorem \ref{thm: conformal mass}. As an  application, we prove positive mass type Theorem \ref{thm:positive mass}.
\hspace{4em}

{\bf Acknowledgment.} The author would like to thank Professor Xingwang Xu for helpful discussions.

\section{Some estimates for normal solutions }\label{section:  estimate for normal }

For brevity, set the notation $\alpha_0$ as follows
$$\alpha_0:=\frac{2}{(n-1)!|\mathbb{S}^n|}\int_{\mr^n}Q(y)e^{nu(y)}\ud y.$$
Given  a measurable set $E$, the average integral for function $\varphi$ over $E$  is defined  as
$$\fint_E\varphi\ud \mu:=\frac{1}{|E|}\int_E\varphi\ud \mu.$$
For $r>0$, the notation $\bar u(r)$ is defined by
$$\bar u(r):=\frac{1}{|\partial B_r(0)|}\int_{\partial B_r(0)}u(x)\ud \sigma(x).$$

The  following lemma is easy but useful in our following   proofs. 
\begin{lemma}\label{lem: mean value}
	For any integer $n\geq 3$, $0< k\leq n-2$ and $r>0,|y|>0$, there holds
	$$\fint_{\partial B_r(0)}\left(\frac{|y|}{|x-y|}\right)^k\ud\sigma(x)\leq 1.$$
\end{lemma}
\begin{proof}
	By H\"older's inequality, one has
	$$\fint_{\partial B_r(0)}\left(\frac{|y|}{|x-y|}\right)^k\ud\sigma(x)\leq \left(\fint_{\partial B_r(0)}\left(\frac{|y|}{|x-y|}\right)^{n-2}\ud\sigma(x)\right)^{\frac{k}{n-2}}.$$
	Thus,  we only need to prove the case $k=n-2$.
	For $n\geq 3$, $|x-y|^{2-n}$ is the Green's function for Laplacian operator. With help of mean value property of harmonic functions, we obtain the formula
	\begin{equation}\label{LL identity}
		\fint_{\partial B_r(0)}\frac{1}{|x-y|^{n-2}}\ud\sigma(x)=\min\{r^{2-n},|y|^{2-n}\}
	\end{equation}
	which has been established in Page 249 of   \cite{Lieb-Loss} by  Lieb and Loss. Using such formula, one has 
	$$\fint_{\partial B_r(0)}\left(\frac{|y|}{|x-y|}\right)^{n-2}\ud\sigma(x)\leq 1.$$ Thus,  we finish the proof.
\end{proof}

\begin{lemma}\label{lem: r^2 Dealtu}
	As $r\to\infty$,  the normal solutions $u(x)$ to \eqref{integral equation} satisfies
	$$r^2\cdot \fint_{\partial B_r(0)}(-\Delta )u\ud \sigma\to(n-2)\alpha_0.$$
\end{lemma}
\begin{proof}
	Firstly, since $Qe^{nu}\in L^1(\mr^n)$, for any $\epsilon>0$, there exists $R_\epsilon>0$ such that 
	\begin{equation}\label{small integral curvature}
		\int_{\mr^n\backslash B_{R_\epsilon}(0)}|Q|e^{nu}\ud y<\epsilon.
	\end{equation}
	With the help of Fubini's theorem, a direct computation yields that 
	\begin{align*}
		&r^2\cdot \fint_{\partial B_r(0)}(-\Delta) u\ud \sigma-(n-2)\alpha_0\\
		=&\frac{2(n-2)}{(n-1)!|\mathbb{S}^n|}\int_{\mr^n}Q(y)e^{nu(y)}\fint_{\partial B_r(0)}\frac{|x|^2}{|x-y|^2}\ud \sigma(x)\ud y\\
		&-\frac{2(n-2)}{(n-1)!|\mathbb{S}^n|}\int_{\mr^n}Q(y)e^{nu(y)}\ud y\\
		=&\frac{2(n-2)}{(n-1)!|\mathbb{S}^n|}\int_{\mr^n}Q(y)e^{nu(y)}\fint_{\partial B_r(0)}\frac{|x|^2-|x-y|^2}{|x-y|^2}\ud \sigma(x)\ud y\\
		=&\frac{2(n-2)}{(n-1)!|\mathbb{S}^n|}\int_{\mr^n}Q(y)e^{nu(y)}\fint_{\partial B_r(0)}\frac{|y|^2-2y\cdot(y-x)}{|x-y|^2}\ud \sigma(x)\ud y\\
		=&\frac{2(n-2)}{(n-1)!|\mathbb{S}^n|}\int_{\mr^n}Q(y)e^{nu(y)}\fint_{\partial B_r(0)}\frac{|y|^2}{|x-y|^2}\ud \sigma(x)\ud y\\
		&-\frac{4(n-2)}{(n-1)!|\mathbb{S}^n|}\fint_{\partial B_r(0)}\int_{\mr^n}Q(y)e^{nu(y)}\frac{y\cdot(y-x)}{|x-y|^2}\ud y\ud \sigma(x).
	\end{align*}
	We  deal with these two terms one by one. Firstly, we split the second term into two parts as follows
	\begin{align*}
		&\fint_{\partial B_r(0)}\int_{\mr^n}\frac{y\cdot(y-x)}{|x-y|^2}Q(y)e^{nu(y)}\ud y\ud\sigma(x)\\
		=&\fint_{\partial B_r(0)}\int_{B_{R_\epsilon}(0)}\frac{y\cdot(y-x)}{|x-y|^2}Q(y)e^{nu(y)}\ud y\ud\sigma(x)\\
		&+\fint_{\partial B_r(0)}\int_{\mr^n\backslash B_{R_\epsilon}(0)}\frac{y\cdot(y-x)}{|x-y|^2}Q(y)e^{nu(y)}\ud y\ud\sigma(x)\\
		=:&II_1(r)+II_2(r).
	\end{align*}
Making use of Fubini's theorem, there holds
$$
		|II_2(r)|\leq C\int_{\mr^n\backslash B_{R_\epsilon}(0)}|Q(y)|e^{nu(y)}\fint_{\partial B_r(0)}\frac{|y|}{|x-y|}\ud\sigma(x)\ud y.
$$
	Applying Lemma \ref{lem: mean value} and the estimate \eqref{small integral curvature}, one has
	\begin{equation}\label{II_2}
		|II_2(r)|\leq C\epsilon.
	\end{equation}
	With the help of dominated convergence theorem,  we have
	\begin{equation}\label{II_1}
		II_1(r)\to 0, \quad \mathrm{as}\quad  r\to\infty.
	\end{equation}
Then,  due to the arbitrary choice of $\epsilon$,  one has 
	\begin{equation}\label{second term}
		\fint_{\partial B_r(0)}\int_{\mr^n}\frac{y\cdot(y-x)}{|x-y|^2}Q(y)e^{nu(y)}\ud y\ud\sigma(x)\to 0
	\end{equation}
	as $r\to\infty.$
	
		As for the first term, we also  split it into two terms as before 
	\begin{align*}
		&\int_{\mr^n}Q(y)e^{nu(y)}\fint_{\partial B_r(0)}\frac{|y|^2}{|x-y|^2}\ud \sigma(x)\ud y\\
		=&\int_{B_{R_\epsilon}(0)}Q(y)e^{nu(y)}\fint_{\partial B_r(0)}\frac{|y|^2}{|x-y|^2}\ud \sigma(x)\ud y\\
		&+\int_{\mr^n\backslash B_{R_\epsilon}(0)}Q(y)e^{nu(y)}\fint_{\partial B_r(0)}\frac{|y|^2}{|x-y|^2}\ud \sigma(x)\ud y\\
		=&:II_{3}(r)+II_4(r).
	\end{align*}
	The dominated convergence theorem yields that 
	$$II_3(r)\to 0, \;\mathrm{as}\quad r\to\infty.$$
With the help of Lemma \ref{lem: mean value} and \eqref{small integral curvature}, one has
	$$|II_4(r)|\leq C\epsilon.$$
	Combining these estimates with  the arbitrary choice of $\epsilon$, there holds
	$$r^2\cdot \fint_{\partial B_r(0)}(-\Delta) u\ud \sigma\to(n-2)\alpha_0, \quad \mathrm{as}\quad r\to\infty.$$
\end{proof}
\begin{lemma}\label{lem:r partial u}
	As $r\to\infty$, the normal solutions $u(x)$ to \eqref{integral equation} satisfies
	\begin{equation}\label{r partial u+alapha_0}
	r	\cdot\bar u'(r) \to-\alpha_0.
	\end{equation}
\end{lemma}
\begin{proof}
A	direct computation yields that 
\begin{align*}
	&x\cdot \nabla u(x)\\
	=&-\frac{2}{(n-1)!|\mathbb{S}^n|}\int_{\mr^n}\frac{x\cdot(x-y)}{|x-y|^2}Q(y)e^{nu(y)}\ud y\\
	=&-\alpha_0+\frac{2}{(n-1)!|\mathbb{S}^n|}\int_{\mr^n}\frac{y\cdot(y-x)}{|x-y|^2}Q(y)e^{nu(y)}\ud y.
\end{align*}
Then, with the help of the estimate \eqref{second term}, one has 
$$
	r\cdot \bar u'(r) =	\fint_{\partial B_r(0)}r\cdot \frac{\partial u}{\partial r}\ud \sigma=	\fint_{\partial B_r(0)}x\cdot \nabla u\ud \sigma \to-\alpha_0
$$
as $r\to\infty.$
\end{proof}

\begin{lemma}\label{lem:r^2nabla u^2}
	As $r\to\infty$,	the normal solutions $u(x)$ to \eqref{integral equation} satisfies
	$$r^2\fint_{\partial B_r(0)}|\nabla u|^2\ud\sigma\to \alpha_0^2.$$
\end{lemma}
\begin{proof}
Making use of Fubini's theorem, 	a direct computation yields that
	\begin{align*}
		&|\nabla u|^2(x)\\
		=&\left(\frac{2}{(n-1)!|\mathbb{S}^n|}\right)^2\int_{\mr^n}\int_{\mr^n}\frac{(x-y)\cdot(x-z)}{|x-y|^2|x-z|^2}Q(y)e^{nu(y)}Q(z)e^{nu(z)}\ud y\ud z\\
		=&\left(\frac{2}{(n-1)!|\mathbb{S}^n|}\right)^2\int_{\mr^n}\int_{\mr^n}\frac{|x-y|^2+|x-z|^2-|y-z|^2}{2|x-y|^2|x-z|^2}Q(y)e^{nu(y)}Q(z)e^{nu(z)}\ud y\ud z\\
		=&\alpha_0\frac{1}{(n-1)!|\mathbb{S}^n|}\int_{\mr^n}\frac{1}{|x-z|^2}Q(z)e^{nu(z)}\ud z\\
		&+\alpha_0\frac{1}{(n-1)!|\mathbb{S}^n|}\int_{\mr^n}\frac{1}{|x-y|^2}Q(y)e^{nu(y)}\ud y\\
		&-\left(\frac{2}{(n-1)!|\mathbb{S}^n|}\right)^2\int_{\mr^n}\int_{\mr^n}\frac{|y-z|^2}{2|x-y|^2|x-z|^2}Q(y)e^{nu(y)}Q(z)e^{nu(z)}\ud y\ud z\\
		=&\frac{-\Delta u(x)}{n-2}\alpha_0-\left(\frac{2}{(n-1)!|\mathbb{S}^n|}\right)^2\int_{\mr^n}\int_{\mr^n}\frac{|y-z|^2}{2|x-y|^2|x-z|^2}Q(y)e^{nu(y)}Q(z)e^{nu(z)}\ud y\ud z.
	\end{align*}
	Taking the same strategy as in Lemma \ref{lem: r^2 Dealtu}, for any $\epsilon>0$, there exists $R_\epsilon>1$ such that
	\begin{equation}\label{small curvature 2}
		\int_{\mr^n\backslash B_{R_\epsilon}(0)}|Q|e^{nu}\ud x<\epsilon.
	\end{equation}

For brevity, set the notation $\ud\mu(y):=|Q(y)|e^{nu(y)}\ud y$ during this  proof. 	A direct computation yields that 
	\begin{align*}
		&\left|\int_{\mr^n}\int_{\mr^n}\frac{|x|^2\cdot|y-z|^2}{2|x-y|^2\cdot|x-z|^2}Q(y)e^{nu(y)}Q(z)e^{nu(z)}\ud y\ud z\right|\\
		\leq &\int_{B_{R_\epsilon}(0)}\int_{B_{R_\epsilon}(0)}\frac{|x|^2\cdot|y-z|^2}{2|x-y|^2\cdot|x-z|^2}\ud \mu(y)\ud\mu(z)\\
		&+\int_{\mr^n\backslash B_{R_\epsilon}(0)}\int_{B_{R_\epsilon}(0)}\frac{|x|^2\cdot|y-z|^2}{2|x-y|^2\cdot|x-z|^2}\ud \mu(y)\ud\mu(z)\\
		&+\int_{B_{R_\epsilon}(0)}\int_{\mr^n\backslash B_{R_\epsilon}(0)}\frac{|x|^2\cdot|y-z|^2}{2|x-y|^2\cdot|x-z|^2}\ud \mu(y)\ud\mu(z)\\
		&+\int_{\mr^n \backslash B_{R_\epsilon}(0)}\int_{\mr^n\backslash B_{R_\epsilon}(0)}\frac{|x|^2\cdot|y-z|^2}{2|x-y|^2\cdot|x-z|^2}\ud \mu(y)\ud\mu(z)\\
		=&:I_1(x)+I_2(x)+I_3(x)+I_4(x).
	\end{align*}
	We are going to deal with these terms one by one. With the  help of dominated convergence theorem, we have
	\begin{equation}\label{I_1}
		I_1(x)\to 0,\quad \mathrm{as}\quad|x|\to\infty.
	\end{equation}
	For $|x|\geq 2R_{\epsilon}$, $|y|\leq R_\epsilon$ and $|z|\geq R_\epsilon$,  it is not hard to check that 
	$$\frac{|x|^2\cdot|y-z|^2}{2|x-z|^2\cdot|x-y|^2}\leq \frac{8|z|^2}{|x-z|^2}.$$
	Then making use of above estimate and Fubini's theorem,  for $|x|\geq 2 R_\epsilon$, one has
	\begin{align*}
		I_2(x)\leq &8 \int_{B_{R_\epsilon}(0)}|Q(y)|e^{nu(y)}\ud y\int_{\mr^n\backslash B_{R_\epsilon}(0)}\frac{|z|^2|Q(z)|e^{nu(z)}}{|x-z|^2}\ud z\\
		\leq & C\int_{\mr^n\backslash B_{R_\epsilon}(0)}\frac{|z|^2|Q(z)|e^{nu(z)}}{|x-z|^2}\ud z.
	\end{align*}
	Then,  for $|x|\geq 2R_\epsilon$, we have
	\begin{align*}
		\fint_{\partial B_r(0)}I_2(x)\ud\sigma(x)\leq &C\fint_{\partial B_r(0)}\int_{\mr^n\backslash B_{R_\epsilon}(0)}\frac{|z|^2|Q(z)|e^{nu(z)}}{|x-z|^2}\ud z\ud\sigma(x)\\
		\leq &C\int_{\mr^n\backslash B_{R_\epsilon}(0)}|Q(z)|e^{nu(z)}\fint_{\partial B_r(0)}\frac{|z|^2}{|x-z|^2}\ud\sigma(x)\ud z.
	\end{align*}
Using Lemma \ref{lem: mean value}, for  $r\geq 2R_\epsilon$,  there holds
	\begin{equation}\label{I_2}
		\fint_{\partial B_r(0)}I_2(x)\ud\sigma(x)\leq C\int_{\mr^n\backslash B_{R_\epsilon}(0)}|Q(z)|e^{nu(z)}\ud z\leq C\epsilon.
	\end{equation}
	Similarly,  one has
	\begin{equation}\label{I_3}
		\fint_{\partial B_r(0)}I_3(x)\ud\sigma(x)\leq C\epsilon.
	\end{equation}

	As for the term $I_4$, one has
	\begin{align*}
		&I_4(x)\\
		=&\int_{\mr^n \backslash B_{R_\epsilon}(0)}\int_{\mr^n\backslash B_{R_\epsilon}(0)}\frac{|x|^2\cdot(|x-y|^2+|x-z|^2-2(x-y)\cdot(x-z))}{2|x-y|^2|x-z|^2}\ud \mu(y)\ud\mu(z)\\
		=&\int_{\mr^n \backslash B_{R_\epsilon}(0)}\ud\mu(y)\int_{\mr^n\backslash  B_{R_\epsilon}(0)}\frac{|x|^2}{2|x-z|^2}\ud \mu(z)\\
		&+\int_{\mr^n \backslash B_{R_\epsilon}(0)}\ud\mu(z)\int_{\mr^n\backslash  B_{R_\epsilon}(0)}\frac{|x|^2}{2|x-y|^2}\ud \mu(y)\\
		&-\int_{\mr^n \backslash B_{R_\epsilon}(0)}\int_{\mr^n\backslash B_{R_\epsilon}(0)}\frac{|x|^2\cdot((x-y)\cdot(x-z))}{|x-y|^2|x-z|^2}\ud \mu(y)\ud\mu(z)
	\end{align*}
	Thus, using Fubini's theorem and H\"older's inequality, one has
	\begin{align*}
		I_4(x)\leq & \epsilon\int_{\mr^n\backslash  B_{R_\epsilon}(0)}\frac{|x|^2}{|x-z|^2}\ud \mu(z)\\
		&+\int_{\mr^n \backslash B_{R_\epsilon}(0)}\int_{\mr^n\backslash B_{R_\epsilon}(0)}\frac{|x|^2}{|x-y|\cdot|x-z|}\ud \mu(y)\ud\mu(z)\\
		=&\epsilon\int_{\mr^n\backslash  B_{R_\epsilon}(0)}\frac{|x|^2}{|x-z|^2}\ud \mu(z)\\
		&+\left(\int_{\mr^n\backslash B_{R_\epsilon}(0)}\frac{|x|}{|x-z|}\ud \mu(z)\right)^2\\
		\leq & \epsilon\int_{\mr^n\backslash  B_{R_\epsilon}(0)}\frac{|x|^2}{|x-z|^2}\ud\mu(z)\\
		&+\int_{\mr^n \backslash B_{R_\epsilon}(0)}\ud\mu(z) \cdot\int_{\mr^n\backslash B_{R_\epsilon}(0)}\frac{|x|^2}{|x-z|^2}\ud \mu(z)\\
		\leq &2\epsilon\int_{\mr^n}\frac{|x|^2}{|x-z|^2}\ud\mu(z).
	\end{align*}
	Following the argument as in Lemma \ref{lem: r^2 Dealtu}, one has 
	$$\fint_{\partial B_r(0)}\int_{\mr^n}\frac{|x|^2}{|x-z|^2}\ud\mu(z)\ud\sigma(x)\leq C.$$
	Thus,  we obtain that
	\begin{equation}\label{I_4}
		\fint_{\partial B_r(0)}I_4(x)\ud\sigma(x)\leq C\epsilon.
	\end{equation}
	
	Combining these estimates \eqref{I_1}, \eqref{I_2}, \eqref{I_3}, \eqref{I_4},  with  the arbitrary choice of $\epsilon$,  there holds
	\begin{equation}\label{I}
	\fint_{\partial B_r(0)}\int_{\mr^n}\int_{\mr^n}\frac{|x|^2\cdot|y-z|^2}{|x-y|^2\cdot|x-z|^2}\ud \mu(y)\ud \mu(z)\ud\sigma(x)\to 0
	\end{equation}
as $r\to\infty$.
	With the help of \eqref{I} and Lemma \ref{lem: r^2 Dealtu}, one has
	$$r^2\fint_{\partial B_r(0)}|\nabla u|^2\ud\sigma\to \alpha_0^2, \quad \mathrm{as}\quad r\to\infty.$$
\end{proof}

\begin{lemma}\label{lem:bar u}
	The normal solutions $u(x)$ to \eqref{integral equation} satisfies
	$$\bar u(r)=(-\alpha_0+o(1))\log r$$
	where $o(1)\to 0$ as $r\to\infty.$
\end{lemma}
\begin{proof}
		Using  Lemma 2.3 in \cite{Li 23 Q-curvature}, there holds
	\begin{equation}\label{u =-aplha log x1}
		u(x)=(-\alpha_0+o(1))\log|x|+\frac{2}{(n-1)!|\mathbb{S}^n|}\int_{B_1(x)}\left(\log\frac{1}{|x-y|}\right)Q(y)e^{nu(y)}\ud y
	\end{equation}
where $o(1)\to 0$ as $|x|\to\infty.$
	For $r\gg1$, 	one has
	\begin{align*}
		&\left|\int_{\partial B_r(0)}\int_{B_1(x)}\left(\log \frac{1}{|x-y|}\right)Q(y)e^{nu(y)}\ud y\ud\sigma(x)\right|\\
		\leq &\int_{\partial B_r(0)}\int_{B_{r+1}(0)\backslash B_{r-1}(0)}\left|\log |x-y|\right|\cdot |Q(y)|e^{nu(y)}\ud y\ud\sigma(x)\\
		\leq &\int_{B_{r+1}(0)\backslash B_{r-1}(0)}|Q(y)|e^{nu(y)}\int_{\partial B_r(0)}\left|\log|x-y|\right|\ud\sigma(x)\ud y\\
		\leq&\int_{B_{r+1}(0)\backslash B_{r-1}(0)}|Q(y)|e^{nu(y)}\int_{\partial B_r(0)\cap B_1(y)}\left|\log|x-y|\right|\ud\sigma(x)\ud y\\
		&+\int_{B_{r+1}(0)\backslash B_{r-1}(0)}|Q(y)|e^{nu(y)}\int_{\partial B_r(0)\cap (\mr^n\backslash B_1(y))}\left|\log|x-y|\right|\ud\sigma(x)\ud y\\
		\leq &\int_{B_{r+1}(0)\backslash B_{r-1}(0)}|Q(y)|e^{nu(y)}\left(C+Cr^{n-1}\log(2r+1)\right)\ud y.
	\end{align*}
	Due to  $Qe^{nu}\in L^1(\mr^n)$, there holds
	$$\fint_{\partial B_r(0)}\int_{B_1(x)}\left(\log \frac{1}{|x-y|}\right)Q(y)e^{nu(y)}\ud y\ud\sigma(x)=o(1)\log r.$$
Combining the above estimate  with  \eqref{u =-aplha log x1}, we finish our proof.
\end{proof}

\begin{lemma}\label{lem:int_B_1(x)}
	The normal solution $u(x)$ to \eqref{integral equation} satisfies 
	$$\frac{1}{|B_1(x)|}\int_{B_1(x)}u(y)\ud y=(-\alpha_0+o(1))\log|x|$$
	where $o(1)\to0$ as $|x|\to\infty.$
\end{lemma}
\begin{proof}
	Making use of Fubini's theorem and the assumption  $Qe^{nu}\in L^1(\mr^n)$, one has
		$$\left|\int_{B_1(x)}\int_{B_1(y)}\log\frac{1}{|z-y|}Q(z)e^{nu(z)}\ud z\ud y\right|\leq C. $$
	Then, using the  estimate \eqref{u =-aplha log x1}, there holds
	\begin{align*}
		\frac{1}{|B_1(x)|}\int_{B_1(x)}u(y)\ud y=(-\alpha_0+o(1))\log|x|.
	\end{align*}
	\end{proof}

 The following lemma plays a central role in \cite{CQY} (see Lemma 3.2), where the proof employs techniques adapted from \cite{Fin}. While our argument essentially follows the approach outlined in \cite{CQY}, we present here a modified proof with some differences.
 \begin{lemma}\label{lem:e^ku=ek bar u}
 	For any $k>0$,   the normal solutions $u(x)$ to \eqref{integral equation} satisfies
 	$$\fint_{\partial B_r(0)}e^{ku}\ud \sigma=e^{k\bar u(r)+o(1)}$$
 	where $o(1)\to 0$ as $r\to\infty.$
 \end{lemma}
 
 \begin{proof}
 	Firstly, we split $u(x)$ into three parts as in \cite{CQY}:
 	\begin{align*}
 		u(x)= &\frac{2}{(n-1)!|\mathbb{S}^n|}\int_{B_{\frac{|x|}{2}}(0)}\log\frac{|y|}{|x|}Q(y)e^{nu(y)}\ud y+C\\
 		&+\frac{2}{(n-1)!|\mathbb{S}^n|}\int_{B_{\frac{|x|}{2}}(0)}\log\frac{|x|}{|x-y|}Qe^{nu}\ud y\\
 		&+\frac{2}{(n-1)!|\mathbb{S}^n|}\int_{\mr^n\backslash B_{\frac{|x|}{2}}(0)}\log\frac{|y|}{|x-y|}Qe^{nu}\ud y\\
 		=&:A_1(x)+A_2(x)+A_3(x).
 	\end{align*}
 	Since $A_1(x)$ is radially symmetric,  for any $x\in \partial B_r(0)$, there holds
 	$$A_1(x)=\fint_{\partial B_r(0)}A_1\ud\sigma.$$
 	As for the term $A_2(x)$, when $|x|\gg1$,  we split it into two terms:
 	\begin{align*}
 		A_2(x)=&\frac{2}{(n-1)!|\mathbb{S}^n|}\int_{B_{\sqrt{|x|}}(0)}\log\frac{|x|}{|x-y|}Q(y)e^{nu(y)}\ud y\\
 		&+\frac{2}{(n-1)!|\mathbb{S}^n|}\int_{B_{\frac{|x|}{2}}(0)\backslash B_{\sqrt{|x|}}(0)}\log\frac{|x|}{|x-y|}Q(y)e^{nu(y)}\ud y.
 	\end{align*}
 	Via a direct computation and the fact $Qe^{nu}\in L^1(\mr^n)$,	one has
 	$$\left|\int_{B_{\sqrt{|x|}}(0)}\log\frac{|x|}{|x-y|}Q(y)e^{nu(y)}\ud y\right|\leq C\log\frac{\sqrt{|x|}}{\sqrt{|x|}-1}$$
 	and 
 	$$\left|\int_{B_{\frac{|x|}{2}}(0)\backslash B_{\sqrt{|x|}}(0)}\log\frac{|x|}{|x-y|}Q(y)e^{nu(y)}\ud y\right|\leq C\int_{B_{\frac{|x|}{2}}(0)\backslash B_{\sqrt{|x|}}(0)}|Q|e^{nu}\ud y.$$
 	Thus, as $|x|\to\infty,$ we obtain that
 	\begin{equation}\label{A_2 to 0}
 		A_2(x)\to 0.
 	\end{equation}

 	Using the fact $|t|\leq e^t+e^{-t}$ and Fubini's theorem, there holds
 	\begin{align*}
 		\left|\fint_{\partial B_r(0)}A_3\ud\sigma\right|\leq &C\int_{\mr^n\backslash B_{|x|/2}(0)}|Q|e^{nu}\fint_{\partial B_r(0)}|\log\frac{|y|}{|x-y|}|\ud\sigma\ud y\\
 		\leq &C\int_{\mr^n\backslash B_{|x|/2}(0)}|Q|e^{nu}\fint_{\partial B_r(0)}\left(\frac{|y|}{|x-y|}+\frac{|x-y|}{|y|}\right)\ud\sigma\ud y\\
 		\leq &C\int_{\mr^n\backslash B_{|x|/2}(0)}|Q|e^{nu}\fint_{\partial B_r(0)}\left(\frac{|y|}{|x-y|}+3\right)\ud\sigma\ud y
 	\end{align*}
 	Using 	Lemma \ref{lem: mean value}, there holds
 	\begin{equation}\label{fint A-3}
 		\left|\fint_{\partial B_r(0)}A_3\ud\sigma\right|\leq C\int_{\mr^n\backslash B_{|x|/2}(0)}|Q|e^{nu}\ud y
 	\end{equation}
 	which tend to zero as $|x|\to\infty$.
 	
 	Regarding the treatment of the term $A_3(x)$, we employ a different approach compared to those presented in \cite{Fin} and \cite{CQY}.
 	We  split $A_3(x)$ as the following two terms:
 	\begin{align*}
 		A_3(x)=&\frac{2}{(n-1)!|\mathbb{S}^n|}\int_{\mr^n\backslash B_{|x|/2}(0)}\log\frac{|y|}{|x-y|}Q^+e^{nu}\ud y\\
 		&+\frac{2}{(n-1)!|\mathbb{S}^n|}\int_{\mr^n\backslash B_{|x|/2}(0)}\log\frac{|x-y|}{|y|}Q^-e^{nu}\ud y\\
 		=&:D_1(x)+D_2(x).
 	\end{align*}
 	For any $p>0$,  letting $|x|\gg1 $ such that 
 	$$\frac{2p\|Q^+e^{nu}\|_{L^1(\mr^n\backslash B_{|x|/2}(0))}}{(n-1)!|\mathbb{S}^n|}\leq n-2.$$
 	Then, if $\|Q^+e^{nu}\|_{L^1(\mr^n\backslash B_{|x|/2}(0))}>0$,  using Jensen's inequality and   Fubini's theorem, we obtain that
 	\begin{align*}
 		&\fint_{\partial B_r(0)}e^{pD_1(x)}	\ud\sigma\\
 		\leq &\fint_{\partial B_r(0)}\int_{\mr^n\backslash B_{|x|/2}(0)}\left(\frac{|y|}{|x-y|}\right)^{\frac{2p\|Q^+e^{nu}\|_{L^1(\mr^n\backslash B_{|x|/2}(0))}}{(n-1)!|\mathbb{S}^n|}}\frac{Q^+e^{nu}\ud y}{\|Q^+e^{nu}\|_{L^1(\mr^n\backslash B_{|x|/2}(0))}}\ud\sigma\\
 		\leq & \int_{\mr^n\backslash B_{|x|/2}(0)}\frac{Q^+e^{nu}}{\|Q^+e^{nu}\|_{L^1(\mr^n\backslash B_{|x|/2}(0))}}	\fint_{\partial B_r(0)}\left(\frac{|y|}{|x-y|}\right)^{\frac{2p\|Q^+e^{nu}\|_{L^1(\mr^n\backslash B_{|x|/2}(0))}}{(n-1)!|\mathbb{S}^n|}}\ud\sigma \ud y.
 	\end{align*}
 	Immediately, Lemma \ref{lem: mean value} yields that 
 	\begin{equation}\label{epII_1}
 		\fint_{\partial B_r(0)}e^{pD_1(x)}	\ud\sigma\leq 1.
 	\end{equation}	
 	For the case $\|Q^+e^{nu}\|_{L^1(\mr^n\backslash B_{|x|/2}(0))}=0$, \eqref{epII_1} holds trivially.
 	
 	Similarly, we only need to deal with $\|Q^-e^{nu}\|_{L^1(\mr^n\backslash B_{|x|/2}(0))}>0$.  There holds
 	\begin{align*}
 		&\fint_{\partial B_r(0)}e^{pD_2(x)}\ud\sigma\\
 		\leq &\int_{\mr^n\backslash B_{|x|/2}(0)}\frac{Q^-e^{nu}}{\|Q^-e^{nu}\|_{L^1(\mr^n\backslash B_{|x|/2}(0))}}	\fint_{\partial B_r(0)}\left(\frac{|x-y|}{|y|}\right)^{\frac{2p\|Q^-e^{nu}\|_{L^1(\mr^n\backslash B_{|x|/2}(0))}}{(n-1)!|\mathbb{S}^n|}}\ud\sigma \ud y\\
 		\leq &3^{\frac{2p\|Q^-e^{nu}\|_{L^1(\mr^n\backslash B_{|x|/2}(0))}}{(n-1)!|\mathbb{S}^n|}}
 	\end{align*}
 	which deduces that
 	\begin{equation}\label{e^pII_2}
 		\lim_{r\to\infty}\sup\fint_{\partial B_r(0)}e^{pD_2(x)}\ud\sigma\leq 1.
 	\end{equation}
 	Using H\"older's inequality and the estimates \eqref{epII_1},\eqref{e^pII_2}, one has
 	$$
 	\lim_{r\to\infty}\sup \fint_{\partial B_r(0)} e^{kA_3(x)}\ud\sigma\leq \lim_{r\to\infty}\sup\left(\fint_{\partial B_r(0)} e^{2kD_1(x)}\ud\sigma\right)^{\frac{1}{2}}\left(\fint_{\partial B_r(0)} e^{2kD_2(x)}\ud\sigma\right)^{\frac{1}{2}}\leq 1.
 	$$
 	Combining the above estimate with the estimates \eqref{A_2 to 0}, \eqref{fint A-3}, we finally obtain that	
 	$$
 	\lim_{r\to\infty}\sup \fint_{\partial B_r(0)} e^{ku-k\bar u}\ud\sigma=	\lim_{r\to\infty}\sup\fint_{\partial B_r(0)} e^{kA_2(x)-k\bar A_2+kA_3-k\bar A_3}\ud\sigma
 	\leq 1.$$ 
 	Using Jensen's inequality, it is easy to see that 
 	$$ \fint_{\partial B_r(0)} e^{ku}\ud\sigma\geq e^{k\bar u}.$$
 	Thus, we obtain that 
 	$$\lim_{r\to\infty}\fint_{\partial B_r(0)} e^{ku-k\bar u}\ud\sigma=1.$$
 	Hence, we finish the proof.
 \end{proof}

 \section{Scalar curvature and Q-curvature integral  }\label{section: scalar curvature and Q-curvature integral}

  As introduced before, to ensure normal metric,  Chang, Qing and Yang  \cite{CQY}  give a geometric criterion  about scalar curvature.    As a generalization,  S.  Wang and Y. Wang \cite{WW} established a  criterion under the control of the integral of negative part of scalar curvature.  For the reader's convenience, we repeat their statement as follows. 
 \begin{lemma}\label{lema: normal}(Theorem 1.4 in \cite{CQY}, Theorem 1.1 in \cite{WW})
 	Given a  conformal metric $g=e^{2u}|dx|^2$ on $\mr^n$ where $n\geq 4$ is an even integer with finite total Q-curvature. If the negative part of scalar curvature  $R_g^-$ satisfies
 	$$	\int_{\mr^n}(R^-_g)^{\frac{n}{2}}e^{nu}\ud x<+\infty,$$
 	then the metric $g$  is normal.
 \end{lemma}

To  prove Theorem \ref{thm: complete R_g geq 0} and Theorem \ref{thm: R_g geq 0}, we just need to combine Theorem \ref{thm: CQY fang NX} with the following result.

\begin{theorem}\label{thm: R_g sign}
	Given a conformal metric $g=e^{2u}|dx|^2$ on $\mr^n$ where $n\geq 4$ is an even integer with finite total Q-curvature.
	\begin{enumerate}
		\item If 	the  scalar curvature $R_g\geq 0$ near infinity,  there holds
		$$0\leq \int_{\mr^n}Qe^{nu}\ud x\leq (n-1)!|\mathbb{S}^n|.$$
	
		\item If the  scalar curvature $R_g(x)\geq C|x|^{-p}$ near infinity for any $p\geq -2$,  there holds
		$$\frac{(n-1)!|\mathbb{S}^n|}{2}\cdot \max\{0, 1-\frac{p}{2}\}\leq \int_{\mr^n}Qe^{nu}\ud x\leq (n-1)!|\mathbb{S}^n|.$$
		
			\item  If 	the  scalar curvature $R_g\geq 0$ near infinity,  there holds
		$$\lim_{|x|\to\infty}\inf\frac{u(x)}{\log|x|}=-\frac{2}{(n-1)!|\mathbb{S}^n|}\int_{\mr^n}Qe^{nu}\ud x.$$
	\end{enumerate}
	
\end{theorem}  
\begin{proof}

Via  a direct computation, the scalar curvature $R_g$ of $g=e^{2u}|dx|^2$ on $\mr^n$ can be written as 
\begin{equation}\label{scalar curvature def}
	R_g=2(n-1)e^{-2u}(-\Delta u-\frac{n-2}{2}|\nabla u|^2).
\end{equation}
 Firstly, for all of  cases, the scalar curvature is supposed  to be  non-negative near infinity and   then Lemma \ref{lema: normal} yields that $u(x)$ is normal.

Case (1): $R_g \geq 0$ near infinity.

With the  help of \eqref{scalar curvature def}, for $r\gg1$, there holds
$$
r^2\fint_{\partial B_r(0)}(-\Delta )u\ud\sigma(x)\geq \frac{n-2}{2}r^2\fint_{\partial B_r(0)}|\nabla u|^2\ud\sigma(x).
$$
By letting $r\to\infty$, Lemma \ref{lem: r^2 Dealtu} and Lemma \ref{lem:r^2nabla u^2}  deduce that 
$$(n-2)\alpha_0\geq \frac{n-2}{2}\alpha_0^2$$
which is equivalent to 
\begin{equation}\label{0 leq alpah_0 leq 2}
	0\leq \alpha_0\leq 2.
\end{equation}
Thus one has 
$$0\leq\int_{\mr^n}Qe^{nu}\ud x\leq (n-1)!|\mathbb{S}^n|.$$

Case (2): $R_g(x) \geq C|x|^{-p}>0$ near infinity. 

By using \eqref{scalar curvature def}, one has
\begin{align*}
	&r^2\fint_{\partial B_r(0)}(-\Delta u)\ud\sigma(x)-\frac{n-2}{2}r^2\fint_{\partial B_r(0)}|\nabla u|^2\ud\sigma(x)\\
	=& \frac{r^2}{2(n-1)}\fint_{\partial B_r(0)}R_ge^{2u}\ud\sigma(x)\\
	\geq &Cr^{2-p}\fint_{\partial B_r(0)}e^{2u}\ud\sigma(x)
\end{align*}
Making use of Jensen's inequality and Lemma \ref{lem:bar u}, there holds
\begin{align*}
	&r^2\fint_{\partial B_r(0)}(-\Delta u)\ud\sigma(x)-\frac{n-2}{2}r^2\fint_{\partial B_r(0)}|\nabla u|^2\ud\sigma(x)\\
	\geq& Cr^{2-p}e^{2\bar u(r)}
	=Cr^{2-p-2\alpha_0+o(1)}.
\end{align*}
If $\alpha_0<1-\frac{p}{2}$, taking $r\to\infty$ in the above inequality, we obtain a contradiction by using Lemma \ref{lem: r^2 Dealtu} and Lemma  \ref{lem:r^2nabla u^2}.  Thus 
we must have 
$$\alpha_0\geq 1-\frac{p}{2}.$$
Combining with \eqref{0 leq alpah_0 leq 2},  we finally  obtain our desired result
$$\frac{(n-1)!|\mathbb{S}^n|}{2}\cdot\max\{0,1-\frac{p}{2}\}\leq \int_{\mr^n}Qe^{nu}\ud x\leq (n-1)!|\mathbb{S}^n|.$$

Case (3): Since $R_g\geq 0$ near infinity, for $|x|\gg1$, 
there holds
$$-\Delta e^{\frac{n-2}{2}u}=R_ge^{\frac{n+2}{2}u}\geq 0, \quad \mathrm{for}\; y\in B_2(x)$$
With the help of mean value theorem for superharmonic  function and Jensen's inequality, one has
\begin{align*}
	e^{\frac{n-2}{2}u(x)}\geq &\frac{1}{|B_1(x)|}\int_{B_1(x)}e^{\frac{n-2}{2}u(y)}\ud y\\
	\geq &e^{\frac{n-2}{2}\frac{1}{|B_1(x)|}\int_{B_1(x)}u(y)\ud y}.
\end{align*}
Making use of Lemma \ref{lem:int_B_1(x)}, one has 
$$u(x)\geq -(\alpha_0+o(1))\log|x|$$
where $o(1)\to 0$ as $|x|\to\infty$. Immediately, there holds
$$\lim_{|x|\to\infty}\inf \frac{u(x)}{\log|x|}\geq -\alpha_0.$$
To show that the equality holds, we notice that 
$$\lim_{|x|\to\infty}\inf \frac{u(x)}{\log|x|}\leq \lim_{r\to\infty}\frac{\bar u(r)}{\log r}=-\alpha_0$$
where the right side comes from Lemma \ref{lem:bar u}.

Thus we finish our proof.
\end{proof} 

\hspace{2em}

{\bf Proof of Theorem \ref{thm: Volume comparison Q geq n-1!}:}

Based on our assumptions, $u(x)$ is normal due to Lemma \ref{lema: normal}. With the help of Theorem \ref{thm: R_g geq 0} and the assumption $Q\geq (n-1)!$, there holds
$$(n-1)!|\mathbb{S}^n|\geq \int_{\mr^n}Qe^{nu}\ud x\geq (n-1)!\int_{\mr^n}e^{nu}\ud x$$
which yields that
\begin{equation}\label{volume upper bound}
	\int_{\mr^n}e^{nu}\ud x\leq |\mathbb{S}^n|.
\end{equation}
On one hand, if the equality holds, we must have $Q\equiv (n-1)!$. With the help of classification theorem for normal solution (see \cite{CY MRL}, \cite{Lin}, \cite{Mar MZ}, \cite{WX} and \cite{Xu05}), there holds
\begin{equation}\label{u rep}
	u(x)=\log \frac{2\lambda}{\lambda^2+|x-x_0|^2}
\end{equation}
for some constant $\lambda>0$ and $x_0\in \mr^n$. On the other hand, if \eqref{u rep} holds, it is easy to check that $R_g\geq 0$ and
$$\int_{\mr^n}e^{nu}\ud x=|\mathbb{S}^n|.$$
Thus, we finish our proof.

\hspace{2em}

{\bf Proof of Theorem \ref{thm:scalar curvature bouended}:}

Firstly, since $R_g\geq 0$ near infinity, Lemma \ref{lema: normal} shows that $g$ is normal.
Using part (1) of Theorem \ref{thm: R_g sign}, we obtain that
\begin{equation}\label{R_g geq 0}
	0\leq \alpha_0\leq 2.
\end{equation}

Now, using the representation \eqref{scalar curvature def} and the assumption $R_g\leq C$,  Lemma \ref{lem:e^ku=ek bar u}   yields that 
\begin{align*}
		&r^2\fint_{\partial B_r(0)}(-\Delta u)\ud\sigma(x)-\frac{n-2}{2}r^2\fint_{\partial B_r(0)}|\nabla u|^2\ud\sigma(x)\\
	\leq& Cr^{2}\fint_{\partial B_r(0)}e^{2u}\ud\sigma\\	
	\leq &Cr^2e^{2\bar u(r)}.
\end{align*}
Then, applying  Lemma \ref{lem:bar u}, one has
\begin{equation}\label{R_g leq C}
	r^2\fint_{\partial B_r(0)}(-\Delta u)\ud\sigma(x)-\frac{n-2}{2}r^2\fint_{\partial B_r(0)}|\nabla u|^2\ud\sigma(x)\leq Cr^{2-2\alpha_0+o(1)}.
\end{equation}
If $\alpha_0>1$, letting $r\to\infty$ in the above inequality \eqref{R_g leq C}, Lemma \ref{lem: r^2 Dealtu} and Lemma \ref{lem:r^2nabla u^2} show that 
$$(n-2)\alpha_0-\frac{n-2}{2}\alpha_0^2\leq 0$$
which yields that $\alpha_0\geq 2$. Combining with \eqref{R_g geq 0}, we show that 
\begin{equation}\label{range alpha_0 for R_g leq C}
	0\leq \alpha_0\leq 1\; \mathrm{or}\; 2.
\end{equation}
Hence, we finish the proof of Part (1).

If $e^{nu}\in L^1(\mr^n)$ in addition, there exists a sequence $r_i \to \infty $ satisfying
\begin{equation}\label{volume finite sequence}
	r_i\int_{\partial B_{r_i}(0)}e^{nu}\ud \sigma \to 0, 
\end{equation}
as $r_i\to\infty$. With the help of \eqref{scalar curvature def} and H\"older's inequality, there holds
\begin{align*}
		&r^2\fint_{\partial B_r(0)}(-\Delta u)\ud\sigma(x)-\frac{n-2}{2}r^2\fint_{\partial B_r(0)}|\nabla u|^2\ud\sigma(x)\\
	\leq& Cr^2\fint_{\partial B_r(0)}e^{2u}\ud\sigma\\
	\leq &Cr^2\left(\fint_{\partial B_r(0)}e^{nu}\ud\sigma\right)^{\frac{2}{n}}\\
	\leq &C\left(r\int_{\partial B_r(0)}e^{nu}\ud\sigma\right)^{\frac{2}{n}}.
\end{align*}
By employing the sequence $r_i$ \eqref{volume finite sequence} chosen earlier together with Lemmas \ref{lem: r^2 Dealtu} and \ref{lem:r^2nabla u^2}, and letting  $r_i\to\infty$, we conclude that
\begin{equation}\label{volume finite}
	(n-2)\alpha_0-\frac{n-2}{2}\alpha_0^2\leq 0.
\end{equation}
With the help of Jensen's inequality and Lemma \ref{lem:bar u}, one has
\begin{align*}
	\int_{B_R(0)}e^{nu}\ud x =&\int^R_0\int_{\partial B_r(0)}e^{nu}\ud \sigma \ud r\\
	\geq &\int^R_0|\partial B_r(0)| e^{n\bar u}\ud r\\
	\geq &C\int^R_0r^{n-1}\cdot r^{-n\alpha_0+o(1)}\ud r.
\end{align*}
Using the above estimate and  the assumption $e^{nu}\in L^1(\mr^n)$, one must have 
\begin{equation}\label{alpha_0 geq 1}
	\alpha_0\geq 1.
\end{equation}
Combining the estimates \eqref{R_g geq 0}, \eqref{volume finite} with \eqref{alpha_0 geq 1}, we finally obtain that 
$$\alpha_0=2.$$

Thus, we finish the proof.

\hspace{2em}

{\bf Proof of Theorem \ref{thm: Volume comparison Q leq n-1!}:}

Firstly, the assumption $R_g\geq 0$ near infinity shows that $g$ is normal due to Lemma \ref{lema: normal}.  If the volume is infinite, the inequality \eqref{volume lower bound} holds automatically. Hence, we only need to deal with the case:  $e^{nu}\in L^1(\mr^n)$. In this situation, Theorem \ref{thm:scalar curvature bouended} shows that
\begin{equation}\label{integral identity}
	\int_{\mr^n}Qe^{nu}\ud x=(n-1)!|\mathbb{S}^n|.
\end{equation}
Since  $Q\leq (n-1)!$, the identity \eqref{integral identity} shows that 
$$\int_{\mr^n}e^{nu}\ud x\geq |\mathbb{S}^n|.$$
If the equality holds, one must have $Q\equiv (n-1)!$.  With the help of classification theorem for normal solution (see \cite{CY MRL}, \cite{Lin}, \cite{Mar MZ}, \cite{WX} and \cite{Xu05}), we finish our proof.

\hspace{2em}

{\bf Proof of Theorem \ref{thm: volume growth under scalar curvature decay}:}

Firstly,  since $R_g \geq 0$ near infinity, Lemma \ref{lema: normal} shows that $g$ is normal. Making use of Theorem \ref{thm: volume entropy} and Theorem \ref{thm: complete R_g geq 0}, we finish the proof.

\hspace{2em}

{\bf Proof of Theorem \ref{thm: integral R_g}:} 

Case (1):

Since $(R^-_g)^{\frac{n}{2}}e^{nu}\in L^1(\mr^n)$, Lemma \ref{lema: normal} yields that $u(x)$ is normal.
Using  \eqref{scalar curvature def} and applying H\"older's inequality, one has
\begin{align*}
&r^2\fint_{\partial B_r(0)}(-\Delta u)\ud\sigma(x)-\frac{n-2}{2}r^2\fint_{\partial B_r(0)}|\nabla u|^2\ud\sigma(x)\\
=& \frac{r^2}{2(n-1)}\fint_{\partial B_r(0)}R_ge^{2u}\ud\sigma(x)\\
\geq &-\frac{r^2}{2(n-1)}\fint_{\partial B_r(0)}R^-_ge^{2u}\ud\sigma(x)\\
\geq &-\frac{r^2}{2(n-1)}\left(\fint_{\partial B_r(0)}(R_g^-)^{\frac{n}{2}}e^{nu}\ud\sigma(x)\right)^{\frac{2}{n}}\\
=&-\frac{|\ms^{n-1}|^{-\frac{2}{n}}}{2(n-1)}\left(r\int_{\partial B_r(0)}(R_g^-)^{\frac{n}{2}}e^{nu}\ud\sigma(x)\right)^{\frac{2}{n}}.
\end{align*}
Based on the assumption \eqref{equ: integral R_g^-}, there exists a sequence $r_i\to\infty $ such that
\begin{equation}\label{sequence r_i}
	r_i \int_{\partial B_{r_i}(0)}(R_g^-)^{\frac{n}{2}}e^{nu}\ud\sigma(x)\to 0.
\end{equation}
Choosing such sequence $r_i$ and making use of  Lemma \ref{lem: r^2 Dealtu} and Lemma  \ref{lem:r^2nabla u^2} as well as \eqref{sequence r_i}, we obtain that
\begin{equation}\label{alpha_0-alpha_0^2 geq 0}
	(n-2)\alpha_0-\frac{n-2}{2}\alpha_0^2\geq 0
\end{equation}
by taking $r_i\to\infty.$ Hence, we obtain that 
$$0\leq \alpha_0\leq 2.$$

Case (2):

Firstly, Lemma \ref{lema: normal} yields that $u$ is normal based on the assumption $R_ge^{2u}\in L^{\frac{n}{2}}(\mr^n)$.
With the help of similar methods as before, one has 
\begin{align*}
	&r^2\fint_{\partial B_r(0)}(-\Delta u)\ud\sigma(x)-\frac{n-2}{2}r^2\fint_{\partial B_r(0)}|\nabla u|^2\ud\sigma(x)\\
	=& \frac{r^2}{2(n-1)}\fint_{\partial B_r(0)}R_ge^{2u}\ud\sigma(x)\\
	\leq &\frac{r^2}{2(n-1)}\fint_{\partial B_r(0)}R^+_ge^{2u}\ud\sigma(x)\\
	\leq &\frac{r^2}{2(n-1)}\left(\fint_{\partial B_r(0)}(R_g^+)^{\frac{n}{2}}e^{nu}\ud\sigma(x)\right)^{\frac{2}{n}}\\
	=&\frac{|\ms^{n-1}|^{-\frac{2}{n}}}{2(n-1)}\left(r\int_{\partial B_r(0)}(R_g^+)^{\frac{n}{2}}e^{nu}\ud\sigma(x)\right)^{\frac{2}{n}}
\end{align*}
where $R^+_g$ is the positive part of scalar curvature. Due to the assumption \eqref{equ: integral R_g}, there exists a sequence $r_j\to\infty$ such that
$$ r_j \int_{\partial B_{r_j}(0)}(R_g^+)^{\frac{n}{2}}e^{nu}\ud\sigma(x)\to 0.$$
By taking $r_j\to\infty$, Lemma \ref{lem: r^2 Dealtu} and Lemma  \ref{lem:r^2nabla u^2} deduce that 
\begin{equation}\label{alpha_0-alpha_0^2 leq 0}
	(n-2)\alpha_0-\frac{n-2}{2}\alpha_0^2\leq 0.
\end{equation}
Combining with \eqref{alpha_0-alpha_0^2 geq 0}, one has
\begin{equation}\label{alpah_0=0or2}
	\alpha_0=0 \quad \mathrm{or}\quad  \alpha_0=2.
\end{equation}
If $e^{nu}\in L^1(\mr^n)$ in addition, the estimate \eqref{alpha_0 geq 1} yields that 
$$\alpha_0=2.$$

Thus, we finish our proof.

\section{Complete manifolds with finitely many simple ends}\label{sec: simple ends}

Given  a complete non-compact four-dimensional manifold $(M,g)$, we say that it has finitely many simple ends if 
$$M=N\cup \left\{\bigcup_{i=1}^kE_i \right\}$$
where $N$ is a compact manifold with boundary
$$\partial N=\bigcup_{i=1}^k\partial E_i$$
and each $E_i$ is a conformally flat simple end satisfying
$$(E_i,g)=(\mr^4\backslash B_1(0),e^{2w_i}|dx|^2)$$
for some  function $w_i$.

Consider the conformal  metric on the end $(\mr^4\backslash B_1(0), e^{2w}|dx|^2)$ where
\begin{equation}\label{equ for w}
	(-\Delta)^{2}w(x)=Q(x)e^{4w(x)}, \quad x\in \mr^4\backslash B_1(0).
\end{equation}
We also say that such metric has finite total Q-curvature if $Qe^{4w}\in L^1(\mr^4\backslash B_1(0))$.
Different from the case on $\mr^4$, we say the solution to \eqref{equ for w} with finite total Q-curvature is normal if $w(x)$ has the following decomposition  
\begin{equation}\label{normal end}
	w(x)=\frac{1}{8\pi^2}\int_{\mr^4\backslash B_1(0)}\log\frac{|y|}{|x-y|}Q(y)e^{4u(y)}\ud y+\alpha_1\log|x|+h(x)
\end{equation}
where $\alpha_1$ is a constant and $h(\frac{x}{|x|^2})$ is some biharmonic function on $B_1(0)$.  When $w$ is normal, such conformal metric $e^{2w}|dx|^2$ is also said to be  normal.
For brevity, set the notation $\alpha_2$ as follows
$$\alpha_2:=\frac{1}{8\pi^2}\int_{\mr^4\backslash B_1(0)}Qe^{4u}\ud y.$$
More details about normal metric on $\mr^4\backslash B_1(0)$  can be found in \cite{CQY2}.

Leveraging the definition of $w$ and the lemmas outlined in the section \ref{section:  estimate for normal }, we are able to  establish the following lemmas for future use.

\begin{lemma}\label{lem:|x|^2nabla^2w}
	For the normal solution $w(x)$ to  \eqref{normal end}, as $r\to\infty$, there holds
	\begin{enumerate}
		\item 	$$r^2\fint_{\partial B_r(0)}|\nabla w|^2\ud\sigma(x)\to\left(\alpha_2-\alpha_1\right)^2,$$
		\item $$r^2\fint_{\partial B_r(0)}(-\Delta)w\ud\sigma(x)\to 2(\alpha_2-\alpha_1).$$
	\end{enumerate}

\end{lemma}

\begin{proof}
	For brevity, we set
	$$w_1(x):=\frac{1}{8\pi^2}\int_{\mr^4\backslash B_1(0)}\log\frac{|y|}{|x-y|}Q(y)e^{4u(y)}\ud y,$$
	and 
	$b(x):=h(\frac{x}{|x|^2}).$
	Thus  $b(x)$ is biharmonic on $B_1(0)$.
	A direct computation yields that
	\begin{equation}\label{partial h}
		\partial_ih(x)=\frac{\partial_ib}{|x|^2}-\sum_{j=1}^4\partial_jb\frac{2x_ix_j}{|x|^4}
	\end{equation}
	which implies that 
	\begin{equation}\label{xpartial h to 0}
		|x|\cdot|\nabla h(x)|\to0 , \quad \mathrm{as}\quad |x|\to\infty.
	\end{equation}
	With the  help of Lemma \ref{lem:r^2nabla u^2} and Lemma \ref{lem:r partial u}, one has
	\begin{equation}\label{r^2w_1^2}
		r^2\fint_{\partial B_r(0)}|\nabla w_1|^2\ud\sigma(x)\to\alpha_2^2
	\end{equation}
	and
	\begin{equation}\label{rw_1}
		\fint_{\partial B_r(0)}x\cdot\nabla w_1\ud\sigma(x)\to-\alpha_2
	\end{equation}
	as $r\to\infty$.
A direct computation yields that 
	\begin{align*}
			&r^2\fint_{\partial B_r(0)}|\nabla w|^2\ud\sigma(x)\\
			= &r^2\fint_{\partial B_r(0)}|\nabla w_1|^2\ud\sigma(x)+\alpha_1^2r^2\fint_{\partial B_r(0)}|\nabla \log|x||^2\ud\sigma(x)+r^2\fint_{\partial B_r(0)}|\nabla h|^2\ud\sigma(x)\\
			&+2\alpha_1r^2\fint_{\partial B_r(0)}\nabla w_1\cdot\nabla\log|x|\ud\sigma(x)+2r^2\fint_{\partial B_r(0)}\nabla w_1\cdot\nabla h\ud\sigma(x)\\
			&+2\alpha_1r^2\fint_{\partial B_r(0)}\nabla h\cdot\nabla\log|x|\ud\sigma(x)\\
			=&r^2\fint_{\partial B_r(0)}|\nabla w_1|^2\ud\sigma(x)+\alpha_1^2+r^2\fint_{\partial B_r(0)}|\nabla h|^2\ud\sigma(x)\\
			&+2\alpha_1\fint_{\partial B_r(0)}x\cdot \nabla w_1\ud\sigma(x)+2r^2\fint_{\partial B_r(0)}\nabla w_1\cdot\nabla h\ud\sigma(x)\\
			&+2\alpha_1\fint_{\partial B_r(0)}x\cdot \nabla h\ud\sigma(x)
	\end{align*}
Using \eqref{xpartial h to 0}, one has 
$$r^2\fint_{\partial B_r(0)}|\nabla h|^2\ud\sigma(x)\to 0, \quad \mathrm{as} \quad r\to\infty$$
and 
$$\fint_{\partial B_r(0)}x\cdot \nabla h\ud\sigma(x)\to 0, \quad \mathrm{as} \quad r\to\infty.$$
Making use of H\"older's inequality and the estimates  \eqref{xpartial h to 0}, \eqref{r^2w_1^2}, one has 
$$\left|r^2\fint_{\partial B_r(0)}\nabla w_1\cdot\nabla h\ud\sigma\right|\leq \left(r^2\fint_{\partial B_r(0)}|\nabla w_1|^2\ud\sigma\right)^{\frac{1}{2}}\left(r^2\fint_{\partial B_r(0)}|\nabla h|^2\ud\sigma\right)^{\frac{1}{2}}\to 0$$
as $r\to\infty.$
Combining these estimates,  we have 
$$r^2\fint_{\partial B_r(0)}|\nabla w|^2\ud\sigma(x)\to (\alpha_2-\alpha_1)^2, \quad \mathrm{as} \quad r\to\infty.$$

	Using the same argument as \eqref{partial h} , it is not hard to check that 
\begin{equation}\label{x^2Delat h=0}
	|x|^2\cdot|\Delta h(x)|\to 0, \quad \mathrm{as}\quad |x|\to\infty.
\end{equation}
With the  help of Lemma \ref{lem: r^2 Dealtu} and the fact $n=4$, one has
\begin{equation}\label{r^2Delta w_1}
	r^2\fint_{\partial B_r(0)}(-\Delta)w_1\ud\sigma\to2\alpha_2,  \quad \mathrm{as} \quad r\to\infty.
\end{equation}
Combining these estimates \eqref{x^2Delat h=0} and \eqref{r^2Delta w_1},  there holds
\begin{align*}
	&r^2\fint_{\partial B_r(0)}(-\Delta)w\ud\sigma(x)\\
	=&r^2\fint_{\partial B_r(0)}(-\Delta)w_1\ud\sigma(x)+r^2\alpha_1\fint_{\partial B_r(0)}(-\Delta)\log|x|\ud\sigma(x)\\
	&+r^2\fint_{\partial B_r(0)}(-\Delta)h\ud\sigma(x)\\
	=&r^2\fint_{\partial B_r(0)}(-\Delta)w_1\ud\sigma(x)-2\alpha_1+r^2\fint_{\partial B_r(0)}(-\Delta)h\ud\sigma(x)\\
	&\longrightarrow2\alpha_2-2\alpha_1,\quad \mathrm{as}\quad r\to\infty.
\end{align*}
Thus,  we finish our proof.
\end{proof}

\begin{lemma}\label{lem: r partial w and bar w}
			For the normal solution $w(x)$ to  \eqref{normal end},  there holds
	\begin{enumerate}
		\item 	$$r\cdot \bar w'(r)\to -\alpha_2+\alpha_1,\quad \mathrm{as}\quad r\to\infty.$$
		\item 	$$\bar w(r)=(-\alpha_2+\alpha_1+o(1))\log r$$
	\end{enumerate}
	where $o(1)\to 0$ as $r\to\infty.$
\end{lemma}
\begin{proof}
	A direct computation yields that
	\begin{align*}
		r\cdot \bar w'(r)=&\fint_{\partial B_r(0)}x\cdot\nabla w_1\ud\sigma(x)+\alpha_1\fint_{\partial B_r(0)}x\cdot\nabla \log|x|\ud\sigma(x)+\fint_{\partial B_r(0)}x\cdot \nabla h\ud\sigma(x)\\
		=&\fint_{\partial B_r(0)}x\cdot\nabla w_1\ud\sigma(x)+\alpha_1+\fint_{\partial B_r(0)}x\cdot \nabla h\ud\sigma(x).
	\end{align*}
By using Lemma \ref{lem:r partial u} and the estimate \eqref{xpartial h to 0}, one has
$$	r\cdot \bar w'(r)\to -\alpha_2+\alpha_1,\quad \mathrm{as}\quad r\to\infty.$$
Since $h(\frac{x}{|x|^2})$ is biharmonic in $B_1(0)$, it is easy to see that 
$$|h(x)|\leq C,\quad \mathrm{as}\quad |x|\to\infty.$$ 
Using such estimate and Lemma \ref{lem:bar u}, there holds
$$\bar w(r)=\bar w_1(r)+\alpha_1\log r+\fint_{\partial B_r(0)}h\ud\sigma=(-\alpha_2+\alpha_1+o(1))\log r.$$
\end{proof}

Similar to  the case for $\mr^n$, Chang, Qing,  and Yang \cite{CQY2} showed that if the scalar curvature is non-negative near infinity,  the conformal metric $e^{2w}|dx|^2$ on $\mr^4\backslash B_1(0)$ is also normal. 
\begin{prop}\label{prop: normal end}(See Proposition 1.12 in \cite{CQY2})
	Given a conformal metric $(\mr^4\backslash B_1(0),e^{2w}|dx|^2)$ with finite total Q-curvature. Supposing  that the scalar curvature  is non-negative near infinity, then the metric $g=e^{2w}|dx|^2$  is normal.
\end{prop}
Set the isoperimetric ratio  for  $(\mr^4\backslash B_1(0),g=e^{2w}|dx|^2)$ as follows
$$I_g(r):=\frac{\left(\int_{\partial B_r(0)}e^{3w}\ud\sigma(x)\right)^{\frac{4}{3}}}{4(2\pi^2)^{\frac{1}{3}}\int_{B_r(0)\backslash B_1(0)}e^{4w}\ud x}.$$
Chang, Qing,  and Yang \cite{CQY2} showed that the limit of such  isoperimetric ratio exists as $r$ tends to infinity if the metric is complete and normal.
\begin{prop}\label{prop in CQY}(Proposition  1.6 and Proposition 1.11 in \cite{CQY2})
	Suppose that  $(\mr^4\backslash B_1(0), e^{2w}|dx|^2)$ is a complete normal metric. Then  there holds
	$$\lim_{r\to\infty}I_g(r)=\lim_{r\to\infty}\left(1+r\cdot\bar w'(r)\right)\geq 0.$$
\end{prop}

With the aid of the aforementioned preparations, we aim to give  the control  of  isoperimetric ratio by constraining the lower bound of scalar curvature.
\begin{lemma}\label{lem: 0 leq I_g leq 1}
Given a  complete metric $g=e^{2w}|dx|^2$ on $\mr^4\backslash B_1(0)$ with finite total Q-curvature.
	If the scalar curvature $R_g\geq 0$ near infinity, then there holds
	$$0\leq\lim_{r\to\infty} I_g(r)\leq 1.$$
\end{lemma}
\begin{proof}
	Firstly, due to $R_g\geq 0$ near infinity, Proposition \ref{prop: normal end} yields that $w$ is normal.
	Recall the definition $R_g$ for  $(\mr^4\backslash B_1(0), e^{2w}|dx|^2)$
	\begin{equation}\label{R_g for n=4}
		R_g=6e^{-2w}\left(-\Delta w-|\nabla w|^2\right).
	\end{equation}
	Based on the assumption that $R_g\geq 0$ near infinity,  for $r\gg1$, there holds
	$$r^2\fint_{\partial B_r(0)}(-\Delta)w\ud\sigma\geq r^2\fint_{\partial B_r(0)}|\nabla w|^2\ud\sigma.$$
	Making use of Lemma \ref{lem:|x|^2nabla^2w}, letting $r\to\infty$, we have
	\begin{equation}\label{alpha_2-alpha_1}
		0\leq \alpha_2-\alpha_1\leq 2.
	\end{equation}
By using Lemma \ref{lem: r partial w and bar w} and Proposition \ref{prop in CQY}, there holds
\begin{equation}\label{1-alpha_2+alpha_1 leq 0}
	\lim_{r\to\infty}I_g(r)=1+\alpha_1-\alpha_2\geq 0.
\end{equation}
Combining with \eqref{alpha_2-alpha_1}, we finally obtain that 
$$0\leq \lim_{r\to\infty}I_g(r)\leq 1.$$
\end{proof}

\begin{lemma}\label{lem: I_g =0}
	Given a  complete metric $g=e^{2w}|dx|^2$ on $\mr^4\backslash B_1(0)$ with finite total Q-curvature.	If $R_g\geq C$ for some constant $C>0$ near infinity, there holds
	$$\lim_{r\to\infty} I_g(r)=0.$$
\end{lemma}
\begin{proof}
	Due to same reason as before, $w$ is normal.
	Using  Jensen's inequality and Lemma \ref{lem: r partial w and bar w},  for $r\gg1$,  one has
	\begin{align*}
		&r^2\fint_{\partial B_r(0)}(-\Delta)w\ud\sigma-r^2\fint_{\partial B_r(0)}|\nabla w|^2\ud\sigma\\
		=&\frac{1}{6}r^2\fint_{\partial B_r(0)}R_ge^{2w}\ud\sigma\\
		\geq &\frac{C}{6}r^2\fint_{\partial B_r(0)}e^{2w}\ud\sigma\\
		\geq &\frac{C}{6}r^2e^{2\bar w(r)}\\
		=&\frac{C}{6}r^{2-2\alpha_2+2\alpha_1+o(1)}.
	\end{align*}
With the  help of Lemma \ref{lem:|x|^2nabla^2w},  letting $r\to\infty$ in the above inequality, we must  have
$$2-2\alpha_2+2\alpha_1\leq 0.$$
Combining with \eqref{1-alpha_2+alpha_1 leq 0},  we finally obtain that
$$\lim_{r\to\infty}I_g(r)=0.$$
\end{proof}

The proof of Theorem \ref{thm: chim(M) R_g geq 0} can be directly derived by combining the key identity established in \cite{CQY2} with Lemma \ref{lem: 0 leq I_g leq 1} and Lemma \ref{lem: I_g =0}.
\begin{theorem}\label{CQY theorem}(See Corollary in \cite{CQY2})
	Given  a complete four-dimensional manifold $(M^4, g)$ with finitely many  conformally flat simple ends and finite total Q-curvature. 
	Suppose that the scalar curvature is non-negative at each end. Then there holds
	$$\chi(M)-\frac{1}{32\pi^2}\int_M\left(|W|^2+4Q_g\right)\ud\mu_g=\sum^k_{i=1}\nu_i$$
	where  in the inverted coordinates centered at each end,
	$$\nu_i=\lim_{r\to\infty}\frac{(vol(\partial B_r))^{\frac{4}{3}}}{4(2\pi^2)^{\frac{1}{3}}vol(B_r\backslash B_1(0))}.$$
\end{theorem}
\begin{remark}
	The quantity $Q_g$ defined in \cite{CQY2} is half the value of the corresponding definition in this paper.
\end{remark}

\section{Conformal mass and normal metric}\label{sec: conformal mass}

In this section, we investigate the equivalence between the conformal mass $m_c(g)$
 admitting a lower bound and the normal  metric.
 
 Before commencing the proof of Theorem \ref{thm: conformal mass}, we present the following lemma. In essence, this result has already been utilized and demonstrated in \cite{CQY}, \cite{WW}, and \cite{Li 23 Q-curvature} more or less. To facilitate readers' comprehension, we provide a repetition of its proof.
\begin{lemma}\label{lem: Delat h}
	If $h(x)$ is a smooth function on $\mr^n$ satisfying $\Delta^mh=0$ for some integer $m\geq 1$,  for any fixed $x\in \mr^n$, there holds
	\begin{equation}\label{growth condition}
		\fint_{B_r(x)}\Delta h\ud y+c\fint_{B_r(x)}|\nabla h|^2\ud y\leq o(1)
	\end{equation}
	where $c$ is a positive constant and $o(1)\to0 $ as $r\to\infty$. Then
	$h(x)$ must be a constant.
\end{lemma}

\begin{proof}
	Firstly, if $m=1$, $h$ is harmonic function. For any integer $i$ with $1\leq  i \leq n$, there holds
	$\Delta\partial_ih=0$. 
Since $\partial_ih$ is harmonic, the mean value property yields that
\begin{equation}\label{patial_i h }
		\partial_ih(x)=\fint_{B_r(x)}\partial_i h\ud y
	\leq (\fint_{B_r(x)}|\nabla h|^2\ud y)^{\frac{1}{2}}.
\end{equation}
Using the assumption \eqref{growth condition} and letting $r\to\infty$, one has 
$\partial_i h\leq 0$. The classical Liouville theorem yields that $\partial_i h$ must be a constant. Hence $|\nabla h|\equiv C$. Inserting it and $\Delta h=0$ into \eqref{growth condition} and letting $r\to\infty$, we obtain that  $|\nabla h|=0$ which yields that $h$ must be a constant.

Secondly, if $m=2$, using \eqref{growth condition} and the fact $c>0$, one has
\begin{equation}\label{Delta h growth}
	\fint_{B_r(x)}\Delta h\ud y\leq o(1).
\end{equation}
Since $\Delta h$ is a harmonic function, making use of \eqref{Delta h growth},  Liouville theorem yields that 
\begin{equation}\label{Delta h=C_1}
	\Delta h\equiv C_1
\end{equation}
 for some non-positive constant $C_1$. Immediately, one  has $\Delta\partial_i h =0$. Same as before, the mean value property shows that 
 \begin{equation}\label{patial_i h2 }
 	\partial_ih(x)=\fint_{B_r(x)}\partial_i h\ud y
 	\leq (\fint_{B_r(x)}|\nabla h|^2\ud y)^{\frac{1}{2}}.
 \end{equation}
Using \eqref{Delta h=C_1} and   \eqref{growth condition}, one has
$$\partial_i h \leq C.$$	
Then,  Liouville theorem shows that $\partial_i h$ is  constant. As a direct computation, there holds $\Delta h=0$.  Thus, we reduce the case $m=2$ to the case $m=1$.  As proved before, $h(x)$ must be a constant.

Finally, if $m>2$, using  Pizzetti’s formula (see Lemma 3 in \cite{Mar MZ}) for polyharmonic functions,
we have
\begin{equation}\label{Pizze formula}
	\fint_{B_r(x)}\Delta h\ud y=\sum^{m-2}_{j=0}c_jr^{2j}\Delta^j(\Delta h)(x)
\end{equation}
	where $c_j$ are all positive constants. Using the condition \eqref{growth condition} and letting $r\to\infty$,  the coefficient of leading term in  \eqref{Pizze formula} must be non-positive i.e.
	\begin{equation}\label{Delta m-1 h leq 0}
		\Delta^{m-1}h \leq 0.
	\end{equation}
Since $\Delta^{m-1} h$ is harmonic, Liouville theorem yields that 
	$$\Delta^{m-1} h=C_2\leq 0$$
	where $C_2$ is a constant. Via a direct computation,  for any $1\leq i\leq n$, one has
	$$\Delta^{m-1}\partial_i h=0.$$
Using Pizzetti's formula again, one has
\begin{equation}\label{Pizz for delta^m-1 partial_i h}
	\fint_{ B_r(x)}\partial_i h\ud y=\sum^{m-2}_{j=0}c_jr^{2j}\Delta^{j}(\partial_i h)(x).
\end{equation}	
Using H\"older's ineqaulity and \eqref{growth condition},  there holds
\begin{equation}\label{grwoth for partial_i h}
	\fint_{B_r(x)}\Delta h\ud y+c\left(\fint_{B_r(x)}\partial_i h\ud y\right)^2\leq o(1).
\end{equation}
Since $m>2$, inserting \eqref{Pizze formula} and \eqref{Pizz for delta^m-1 partial_i h} into \eqref{grwoth for partial_i h},  the leading term of left side in \eqref{grwoth for partial_i h}	is 
$$c\left(c_{m-2}(\Delta^{m-2}\partial_i h)(x)\right)^2 r^{4m-8}.$$
Thus, one must have 
$$	\Delta^{m-2}\partial_i h=0.$$ Due to the arbitrary choice of $i$ and a direct computation,  there holds
$$\Delta^{m-1} h=0.$$
Using a standard inductive method, we finally obtain that $h(x)$ must be a constant. 
\end{proof}

{\bf Proof of Theorem \ref{thm: conformal mass}:}

Firstly, if the conformal mass $m_c(g)>-\infty$,  for any $r\geq 1$ and  fixed  $x\in\mr^n$, one has
\begin{equation}\label{partial B_r Re^2u geq -Cr^n-3}
	\int_{\partial B_r(x)}R_ge^{2u}\ud \sigma\geq  -Cr^{n-3}.
\end{equation} 
Then integrating \eqref{partial B_r Re^2u geq -Cr^n-3} from $1$ to $r$, there holds
$$	\int_{B_r(x)\backslash B_1(x)}R_ge^{2u}\ud x\geq -Cr^{n-2}.$$
Immediately, based on the smoothness assumption on $u(x)$,  for any $r\geq 1$, one has 
\begin{equation}\label{intR_g on B_r}
	\int_{B_r(x)}R_ge^{2u}\ud x\geq -Cr^{n-2}.
\end{equation}
For brevity, set the following notations
$$v(x):=\frac{2}{(n-1)!|\mathbb{S}^n|}\int_{\mr^n}\log\frac{|y|}{|x-y|}Q(y)e^{nu(y)}\ud y$$ and
 $$h(x):=u(x)-v(x).$$
Thus, $h(x)$ is a polyharmonic function satisfies
\begin{equation}\label{polyharmonic h}
	(-\Delta)^{\frac{n}{2}}h=0.
\end{equation}

With the help of Lemma 2.11 in \cite{Li 23 Q-curvature}, for $r>2|x|+1$, there holds
\begin{equation}\label{Delat v+nabla ^2v}
	\int_{B_r(x)}\left(|\Delta v|+|\nabla v|^2\right)\ud y\leq \int_{B_{2r}(0)}\left(|\Delta v|+|\nabla v|^2\right)\ud y\leq Cr^{n-2}.
\end{equation}
Using the representation of scalar curvature \eqref{scalar curvature def} and the estimate  \eqref{intR_g on B_r}, we obtain that 
\begin{equation}\label{Dealt u+nabla^2u}
	\int_{B_r(x)}\left(\Delta u+\frac{n-2}{2}|\nabla u|^2\right)\ud y\leq Cr^{n-2}.
\end{equation}
With the help of the estimates \eqref{Delat v+nabla ^2v} and \eqref{Dealt u+nabla^2u}, there holds
\begin{equation}\label{Delata h+nabla^2h}
		\int_{B_r(x)}\left(\Delta h+\frac{n-2}{4}|\nabla h|^2\right)\ud y\leq Cr^{n-2}.
\end{equation}
Making use of Lemma \ref{lem: Delat h}, $h(x)$ must be a constant. Hence, $u(x)$ is normal.

Conversely, suppose that $u(x)$ is normal. Through a straightforward change of coordinates, for any fixed $x\in\mr^n$, it follows from Lemma \ref{lem: r^2 Dealtu} and Lemma \ref{lem:r^2nabla u^2} that
$$r^2\fint_{\partial B_r(x)}\Delta u\ud \sigma +\frac{n-2}{2}r^2\fint_{\partial B_r(x)}|\nabla u|^2 \ud \sigma \to-(n-2)\alpha_0+\frac{n-2}{2}\alpha_0^2$$
as $r\to\infty.$
Using \eqref{scalar curvature def}, the above estimate is equivalent to 
$$\frac{1}{(n-1)(n-2)|\mathbb{S}^{n-1}|}r^{n-3}\int_{\partial B_r(x)}R_ge^{2u}\ud\sigma\to \alpha_0(2-\alpha_0)$$
as $r\to\infty.$ Since the choice of  $x$ is arbitrary, one has
$$m_c(g)=\alpha_0(2-\alpha_0).$$

Thus, we finish the proof.

\hspace{2em}

Combining Theorem \ref{thm: conformal mass} with the  Pohozaev-type identity  for Q-curvature (see Theorem 1.1 in \cite{Xu05}), we have the following  interesting corollary.
\begin{corollary}\label{cor: PH identity}
	Given a  conformal metric $g=e^{2u}|dx|^2$ on $\mr^n$ where $n\geq 4$ is an even integer with finite total Q-curvature. Supposing that the metric $g$ is normal, then there holds
	$$m_c(g)=-\frac{4}{n!|\mathbb{S}^n|}\int_{\mr^n}x\cdot \nabla Q e^{nu}\ud x.$$
\end{corollary}

\hspace{2em}

{\bf Proof of Theorem \ref{thm:positive mass}:}

When the  complete metric is normal, Theorem 1.3 in \cite{CQY} (See also \cite{Fa}, \cite{NX})  shows that 
\begin{equation}\label{normal Q leq 1/2}
\int_{\mr^n}Qe^{nu}\ud x\leq \frac{(n-1)!|\mathbb{S}^n|}{2}.
\end{equation}
Based on the assumption $Q\geq0$ and the estimate \eqref{normal Q leq 1/2}, one has
\begin{equation}\label{ 0 leq Q leq 1/2}
	0\leq \int_{\mr^n}Qe^{nu}\ud x\leq \frac{(n-1)!|\mathbb{S}^n|}{2}.
\end{equation}
Using the  above estimate and Theorem \ref{thm: conformal mass}, we show that 
$$m_c(g)\geq 0.$$
On one hand, when the equality holds, one must have
$$ \int_{\mr^n}Qe^{nu}\ud x=0,$$
which yields that $Q\equiv 0$.  Then using the normal metric assumption and \eqref{integral equation}, we finally obtain that $u$ must be a constant. 
On the other hand, when $u$ is a constant, it is obvious to see that $m_c(g)=0$. Thus,  we finish our proof.

\end{document}